\documentclass{amsart}
\usepackage{amssymb}

\newtheorem{theorem}{Theorem}[section]
\newtheorem{lemma}[theorem]{Lemma}

\newtheorem{corollary}[theorem]{Corollary}
\newtheorem{setup}[theorem]{}

\newtheorem{question}[theorem]{Question}
\newtheorem{_definition}[theorem]{Definition}

\newtheorem{_remark}[theorem]{Remark}
\newenvironment{remark}{\begin{_remark}\rm}{\end{_remark}}

\newtheorem{_claim}[theorem]{Claim}
\newenvironment{claim}{\begin{_claim}\rm}{\end{_claim}}
\numberwithin{equation}{section}
\numberwithin{table}{section}
\numberwithin{figure}{section}

\newcommand{\B}{\mathord{\mathbb B}}
\newcommand{\C}{\mathord{\mathbb C}}
\newcommand{\F}{\mathord{\mathbb F}}

\newcommand{\N}{\mathord{\mathbb N}}
\renewcommand{\P}{\mathord{\mathbb P}}
\newcommand{\Q}{\mathord{\mathbb Q}}
\newcommand{\R}{\mathord{\mathbb R}}
\newcommand{\Z}{\mathord{\mathbb Z}}

\newcommand{\KK}{\mathord{\mathcal K}}

\newcommand{\OO}{\mathord{\mathcal O}}

\newcommand{\alb}{{\text{\rm alb}}}
\newcommand{\Alb}{{\text{\rm Alb}}}

\newcommand{\Aut}{\text{\rm Aut}}
\newcommand{\Bir}{{\text{\rm Bir}}}

\newcommand{\Gal}{\text{\rm{Gal}}}
\newcommand{\Hom}{\text{\rm Hom}}
\newcommand{\id}{\text{\rm id}}
\newcommand{\Imm}{\text{\rm Im}}
\newcommand{\Ker}{\text{\rm Ker}}
\newcommand{\MRC}{{\text{\rm MRC}}}

\newcommand{\Nef}{\text{\rm Nef}}
\newcommand{\NS}{\text{\rm NS}}
\newcommand{\ord}{\text{\rm ord}}

\newcommand{\Pic}{\text{\rm Pic}}

\newcommand{\rank}{\text{\rm rank}}

\newcommand{\torsion}{\text{\rm torsion}}

\newcommand{\variety}{{\text{\rm variety}}}
\newcommand{\vol}{{\text{\rm vol}}}

\newcommand{\SL}{\mathord{\hbox{\sl{SL}}}}

\begin{document}
\begin{large}
\title
[Dynamics of automorphisms] {Dynamics of automorphisms on projective
complex manifolds}
\author{De-Qi Zhang}
\address
{
\textsc{Department of Mathematics} \endgraf
\textsc{National University of Singapore, 2 Science Drive 2,
Singapore 117543, Singapore
\endgraf
and
\endgraf
Max-Planck-Institut f\"ur Mathematik,
Vivatsgasse 7,
53111 Bonn, Germany
}}
\email{matzdq@nus.edu.sg}

\begin{abstract}
We show that the dynamics of automorphisms on all projective complex
manifolds $X$ (of dimension 3, or of any dimension but assuming the Good
Minimal Model Program or Mori's Program) are canonically built up
from the dynamics on just three types of projective complex
manifolds:
complex tori, weak Calabi-Yau manifolds and
rationally connected manifolds.
As a by-product, we confirm
the conjecture of Guedj \cite{Gu06} for automorphisms on
3-dimensional projective manifolds, and also determine $\pi_1(X)$.

\end{abstract}

\subjclass{14J50, 14E07, 32H50}
\keywords{automorphism of variety, Kodaira dimension, topological entropy}

\maketitle

\section{Introduction}

We work over the field $\C$ of complex numbers.

\par
We show that the dynamics of automorphisms on all projective complex
manifolds (of dimension 3, or of any dimension but assuming the Good
Minimal Model Program or Mori's Program) are canonically built up
from the dynamics on just three types of projective complex
manifolds:
complex tori, weak Calabi-Yau manifolds, and rationally
connected manifolds.

\par
For a similar phenomenon on the dynamics in dimension 2, we refer to \cite{Ca99}.
Here a projective manifold $X$ is {\it weak Calabi-Yau} or simply
{\it wCY} if the Kodaira dimension $\kappa(X) = 0$ and the
first Betti number $B_1(X) = 0$.
A projective manifold $X$ is {\it rationally connected}
(the higher dimensional analogue of a rational surface) if any two points
on $X$ are connected by a rational curve; see \cite{Cp} and \cite{KoMM}.
For a smooth projective surface $X$, it is wCY if and only if $X$ itself or its
\'etale double cover is birational to a K3 surface, while $X$
is rationally connected if and only if it is a rational surface.
\par
For the recent development on complex dynamics, we refer to the
survey articles \cite{Fr06} and \cite{Zs} and the references therein.
See also \cite{Ca01}, \cite{DS05}, \cite{McP2}, \cite{NZ}
and \cite{ICCM}.
For algebro-geometric approach to dynamics of automorphisms due to
Oguiso, see \cite{Og03}, \cite{Og04} and \cite{Og06}.
\par
We shall consider dynamics of automorophisms on projective complex
manifolds of dimension $\ge 3$. To focus on the dynamics of
genuinely high dimension, we introduce the notions of {\it rigidly
parabolic} pairs $(X, g)$ and pairs $(X, g)$ of {\it primitively
positive entropy}, where $X$ is a projective manifold of
$\dim X \ge 3$, and $g \in \Aut(X)$. In other words, these are the
pairs where the dynamics are not coming from the dynamics of lower
dimension; see Convention \ref{entropy}, and also Lemma \ref{ps} for
the classification of the rigidly parabolic pairs in the case of
surfaces. These notions might be the geometrical incarnations of
McMullen's lattice-theoretical notion ''essential lattice isometry''
in \cite{Mc02} \S 4. By the way, all surface automorphisms of
positive entropy are automatically primitive.
\par
In Theorem \ref{Thk} below and Theorem \ref{Thq} of \S 3, it is shown that a pair
$(X, g)$ of rigidly parabolic or of primitively positive entropy
exists only when the Kodaira dimension $\kappa(X) \le 0$ and
irregularity $q(X) \le \dim X$. If $(Y, g)$ is just of positive
entropy, then one can only say that $\kappa(Y) \le \dim Y - 2$ which
is optimal by Lemma \ref{k-2}.
\par
In Theorems \ref{Thk=0} and \ref{Thk=-1} of \S 3, we determine the
$g$-structure of manifolds $X$ of dimension $\ge 3$, with $\kappa(X)
= 0$ and $-\infty$, respectively. The difficult part in Theorem
\ref{Thk=0} is to show that the regular action of $g$ on the initial
manifold $X$ is {\it equivariant} to a nearly regular action on
another better birational model $X'$; see Convention
\ref{entropy}. Such difficulty occurs only in dimension $\ge 3$ due
to the high-dimensional new phenomenon of non-uniqueness and
non-smoothness of minimal models of a birational class. Recall that
a variety $Z$ with mild terminal singularities is {\it minimal} if
the canonical divisor $K_Z$ is nef (= numerically effective).
\par
We now state the main results.
The result below says that dynamics occur essentially
only on those $X$ with Kodaria dimension $\kappa(X) \le 0$.

\par
In the following, $(X, g)$ is of {\it primitively positive entropy} if $g$ is of positive entropy but
$g$ is not induced from an automorphism on a manifold of
lower dimension; $(X, g)$ is {\it rigidly parabolic} if $g$ is parabolic and every descent of
$g$ to a lower and positive dimensional manifold is parabolic,
while an automorphism $h$ is {\it parabolic} if it is of null entropy and $\ord(h) = \infty$;
see \ref{entropy} for the precise definitions.

\begin{theorem}\label{Thk}
Let $X$ be a projective complex manifold
of $\dim X \ge 2$, and with
$g \in \Aut(X)$. Then we have:
\begin{itemize}
\item[(1)] Suppose that $(X, g)$ is either rigidly parabolic
or of primitively positive entropy $($see $(\ref{entropy})$$)$. Then
the Kodaira dimension $\kappa(X) \le 0$.
\item[(2)]
Suppose that $\dim X = 3$ and $g$ is of positive entropy.
Then $\kappa(X) \le 0$, unless $d_1(g^{-1}) = d_1(g)= d_2(g) = e^{h(g)}$
and it is a Salem number. Here $d_i(g)$ are
dynamical degrees and $h(g)$ is the entropy
$($see $(\ref{entropy}))$.
\end{itemize}
\end{theorem}

If $(X, g)$ does not satisfy the conditions in Theorem \ref{Thk}(1), then a positive power of $g$ is induced from some automorphism on a manifold of lower dimension.
So by Theorem \ref{Thk} and by the induction on the dimenion, we have only to treat the dynamics on
those $X$ with $\kappa(X) = 0$ or $-\infty$. This is done in
Theorems \ref{Thk=0} and \ref{Thk=-1}. See the statements in \S 3
for details.

\par
As sample results, we now give some applications to our results in \S 3 in the
case of threefolds.
\par
The result below says that $3$-dimensional dynamics of positive
entropy (not necessarily primitive) are just those of $3$-tori, weak
Calabi-Yau $3$-folds and rational connected $3$-folds, unless
dynamical degrees are Salem numbers.

\begin{theorem}\label{3k}
Let $X'$ be a smooth projective complex threefold. Suppose that $g
\in \Aut(X')$ is of positive entropy.
Then there is a pair $(X, g)$ birationally equivariant to $(X', g)$,
such that one of the cases below occurs.
\begin{itemize}
\item[(1)]
There are a $3$-torus $\tilde{X}$ and
a $g$-equivariant \'etale Galois cover $\tilde{X} \rightarrow X$.
\item[(2)]
$X$ is a weak Calabi-Yau threefold.
\item[(3)]
$X$ is a rationally connected threefold in the sense of \cite{Cp} and \cite{Ko}.
\item[(4)]
$d_1(g^{-1} | X) = d_1(g | X)= d_2(g | X) = e^{h(g | X)}$
and it is a Salem number.
\end{itemize}
\end{theorem}

\par
The higher dimensional analogue of the result above is summarized in
Theorems \ref{Thk=0} and \ref{Thk=-1} in \S 3, and it confirms
the claim about the building blocks of dynamics made in the abstract of the paper.
In the results there,
we need to assume the existence of Good Minimal Models for non-uniruled
varieties. The recent breakthrough in \cite{BCHM} on the existence of flips
and the finite generation of canonical rings suggests
that such existence problem of a {\it usual}
minimal model is quite within the reach in a near future,
but the question about the Goodness of a minimal model in dimension $\ge 4$
might be much more difficult than the same question in dimension three
affirmatively answered by Kawamata and Miyaoka.

\par
The following confirms the conjecture of Guedj \cite{Gu06} page 7
for automorphisms on 3-dimensional projective manifolds;
for the general case, see \cite[Appendix A]{NZ} or \cite[Theorem 1.1]{CH}.

\begin{theorem}\label{Thg}
Let $X$ be a smooth projective complex threefold admitting
a cohomologically hyperbolic automorphism $g$ in the sense of
\cite{Gu06} page $3$. Then the Kodaira dimension $\kappa(X) \le 0$.
More precisely, either $X$ is a weak Calabi-Yau threefold,
or $X$ is rationally connected, or there is a $g$-equivariant
birational morphism $X \to T$ onto a $Q$-torus.
\end{theorem}

We can also determine the fundamental group below.
For the case of $\kappa(X) = 0$ we refer to Namikawa-Steenbrink
\cite[Corollary (1.4)]{NS}.

\begin{theorem}\label{Thp}
Let $X$ be a smooth projective complex threefold with $g \in \Aut(X)$
of primitively positive entropy. Suppose that the Kodaira dimension
$\kappa(X) \ne 0$. Then either $\pi_1(X) = (1)$, or $\pi_1(X) = \Z^{\oplus 2}$.
\end{theorem}

\par \vskip 1pc
For examples $(X, g)$ of positive entropy with $X$ a torus (well
known case), or a rational manifold (take product of rational
surfaces), or a Calabi-Yau manifold, we refer to \cite{Ca01},
\cite{McK3} and \cite{McP2}, and Mazur's example of
multi-degree two hypersurfaces in $\P^1 \times \cdots \times \P^1$ as
in \cite{DS05} Introduction.
See also Remark \ref{remark1.1} below.

\par \vskip 1pc
We define the following sets of dynamical degrees for automorphisms
of positive entropy, where $X$ is $wCY =$ weak Calabi-Yau if
$\kappa(X) = 0 = B_1(X)$, where rat.conn. $=$ rationally connected
is in the sense of \cite{Cp} and \cite{KoMM}, where type (*) = type
(t) (torus), or type (cy) (weak Calabi-Yau), or type (rc) (rational
connected). Note that $D_1^*(n) \subset D_1^*(n)'$.
$$D_1^t(n) := \{\lambda \in \R_{> 1} \ | \
\lambda = d_1(g) \,\, \text{\rm for} \,\, g \in \Aut(X) \,\,
\text{\rm with} \,\, X \,\, \text{\rm an {\it n}-torus} \},$$
$$D_1^{cy}(n) := \{\lambda \in \R_{> 1} \ | \
\lambda = d_1(g) \,\, \text{\rm for} \,\, g \in \Aut(X) \,\,
\text{\rm with} \,\, X \,\, \text{\rm a wCY {\it n}-fold} \},$$
$$D_1^{rc}(n) := \{\lambda \in \R_{> 1} \ | \
\lambda = d_1(g) \,\, \text{\rm for} \,\, g \in \Aut(X) \,\,
\text{\rm with} \,\, X \,\, \text{\rm a rat.conn. {\it n}-fold} \},$$
$$Sa = \{\lambda \in \R_{> 1} \ | \
\lambda \,\,\text{\rm is a Salem number}\}.$$
Denote by $D_1^{*}(n)'$ the set of those $\lambda \in \R_{>1}$ satisfying the following:
there are a type (*) $n$-fold $X$, an ample Cartier integral divisor $H$ on $X$,
a sublattice $L \subset \NS(X)/(\torsion)$ containing $H$ and a $\sigma \in \Hom_{\Z}(L, L)$
which is bijective and preserves the induced multi product form on $L$ such that
$\sigma^*P \equiv \lambda P$ for a nonzero nef $\R$-divisor $P$ in $L \otimes_{\Z} \R$.

\par \vskip 1pc
We conclude the introduction with the result and question below which
suggest a connection between the existence of meaningful dynamics and the
theory of algebraic integers like the dynamical degrees $d_i(g)$.
\par
See McMullen \cite{Mc02} for the realization of some Salem numbers
as dynamical degrees of $K3$ automorphisms.

\begin{theorem}\label{Thsalem}
Let $X$ be a smooth projective complex threefold.
Suppose $g \in \Aut(X)$ is of positive entropy.
Then the first dynamical degree satisfies
$$d_1(g) \,\, \in \,\, Sa \, \cup \,  D_1^t(3) \, \cup \, D_1^{cy}(3) \, \cup \, D_1^{rc}(3).$$
Further, for some $s > 0$,
$$d_1(g^s) \,\, \in \,\, D_1^{rc}(2)' \,\,\cup_{k=2}^{3} \,\,
\{D_1^t(k) \, \cup \, D_1^{cy}(k) \, \cup \, D_1^{rc}(k)\}.$$
\end{theorem}

The question below has a positive answer in dimension 2;
see \cite{Ca99}.

\begin{question}\label{question}
Let $X$ be a projective complex manifold of dimension $n \ge
2$ and $g \in \Aut(X)$ of primitively positive entropy. Does the first dynamical
degree $d_1(g)$ satisfy the following
$$d_1(g) \,\, \in \, \, \cup_{k=2}^{n} \, \,
\{D_1^t(k) \, \cup \, D_1^{cy}(k) \, \cup D_1^{rc}(k)\} \, ?$$
\end{question}

\begin{remark}\label{remark1.1}
Take the elliptic curve $E_i := \C/(\Z + \Z\sqrt{-1})$.
Following Igusa's construction (see Oguiso - Sakurai \cite[(2.17)]{OS},
or Ueno \cite{Ue}), one can construct free action of
$\Gamma : = (\Z/(2))^{\oplus 2}$ on the abelian
variety $A := E_i \times E_i \times E_i$ such that
$X : = A/\Gamma$ is a smooth Calabi-Yau variety with $K_X \sim 0$.
The group $\SL_3(\Z)$ acts on $A$, and
as observed in \cite[(4.5)]{DS04}, it contains a free abelian subgroup $G$
of rank $2$ such that the action on $A$ by each $\id \ne g \in G$ is of positive entropy.
If we can find such a $g$ normalizing $\Gamma$, then $g | A$ descends to
a $\bar{g} \in \Aut(X)$ of positive entropy.
\end{remark}

\begin{remark}
Like \cite{Og03} - \cite{Og06}, \cite{KOZ}, \cite{NZ2},
\cite{Zh} and \cite{CH},
our approach is algebro-geometric in nature.
\end{remark}

\par \vskip 1pc \noindent
{\bf Acknowledgement.}
\par
I would like to thank

Tien-Cuong Dinh and Nessim Sibony for the explanation of the paper \cite{DS04}
and the reference \cite{Di},
Noboru Nakayama for answering questions on nef and big divisors, Keiji
Oguiso for sending me the papers \cite{Og03}-\cite{Og06} on dynamics
of automorphism groups on projective varieties, Eckart Viehweg for patiently
answering my persistent questions on moduli spaces and isotrivial
families, Kang Zuo for the explanation on families of Calabi-Yau
manifolds, and the referee for the suggestions for the improvement of the paper.
\par
I also like to thank the Max Planck Institute for Mathematics at Bonn
for the warm hospitality in the first quarter of the year 2007.
This project is supported by an Academic Research Fund of NUS.

\section{Preliminary results}

In this section we recall definitions and prove some lemmas.
\par
Results like Lemmas \ref{ps} and \ref{nefbig}
are not so essentially used by the paper, but are hopefully of independent
interest.

\begin{setup} {\bf Conventions and remarks}
\label{entropy}
\end{setup}

(1) For a linear transformation $T : V \rightarrow V$ of a vector space
$V$, let $\rho(T)$ be the {\it spectral radius} of $T$, i.e., the
largest modulus of eigenvalues of $T$.

\par
(2) We shall use the terminology and notation in
Kawamata-Matsuda-Matsuki \cite{KMM} and
Koll\'ar - Mori \cite{KM}. Most of the divisors are $\R$-Cartier
divisors: $\sum_{i=1}^s r_i D_i$ with $r_i \in \R$ and $D_i$ Cartier
prime divisor.
\par
Let $X$ be a projective  manifold. Set $H^*(X, \C) =
\oplus_{i=0}^{2n} H^i(X, \C)$. There is the Hodge decomposition:
$$H^k(X, \C) = \oplus_{i+j=k} H^{i,j}(X, \C).$$
Denote by $H^{i,i}_a(X, \C)$ the subspace of $H^{i,i}(X, \C)$
spanned by the algebraic
subvarieties of complex codimension $i$.

\par
(3) Let $\Pic(X)$ be the {\it Picard group}, and $\NS(X) =
\Pic(X)/($algebraic equivalence) $= H^1(X, \OO_X^*)/\Pic^0(X)
\subseteq H^2(X, \Z)$ the {\it Neron-Severi group}. $\NS(X)$ is a
finitely generated abelian group whose rank is the {\it Picard
number}.
\par
Set $\NS_{\B}(X) = \NS(X) \otimes_{\Z} \B \subset H^2(X, \B)$ for
$\B = \Q$ and $\R$. Let $\Nef(X)$ be the closed cone of nef divisors
in $\NS_{\R}(X)$. So $\Nef(X)$ is the closure of the ample cone.
Also $\Nef(X) \subseteq \overline{\KK(X)}$, the closure of the
K\"ahler cone $\KK(X)$.
\par
Let $N_1(X)$ be the $\R$-space generated by algebraic $1$-cycles
modulo numerical equivalence; see \cite{KM} (1.16). When $X$ is a
surface, $N_1(X) = \NS_{\R}(X)$.

\par
(4) Let $g \in \Aut(X)$. Denote by $\rho(g)$ the {\it spectral
radius} of $g^*| H^*(X, \C)$. It is known that either $\rho(g) > 1$,
or $\rho(g) = 1$ and all eigenvalues of $g^* | H^*(X, \C)$ are of
modulus $1$. When $\log \rho(g) > 0$ (resp. $\log \rho(g) = 0$) we
say that $g$ is of {\it positive entropy} (resp. {\it null
entropy}).

\par
We refer to Gromov \cite{Gr}, Yomdin \cite{Yo}, Friedland
\cite{Fr95}, and Dinh - Sibony \cite{DS05} page 302, for the
definition of the $i$-th {\it dynamical degree} $d_i(g)$ for $1 \le
i \le n = \dim X$ (note that $d_n(g) = 1$ now and set $d_0(g) = 1$)
and the actual definition of the {\it topological entropy} $h(g)$
which turns out to be $\log \rho(g)$ in the setting of our paper.

\par
(5) Let $Y$ be a projective variety and $g \in \Aut(Y)$. We say that
$g$ is of {\it positive entropy}, or {\it null entropy}, or {\it
parabolic}, or {\it periodic}, or {\it rigidly parabolic}, or of
{\it primitively positive entropy} (see the definitions below), if
so is $g \in \Aut(\tilde{Y})$ for one (and hence all) $g$-equivariant resolutions
as guaranteed by
Hironaka \cite{Hi}. The definitions do not depend on the choice of
$\tilde{Y}$ because every two $g$-equivariant resolutions are
birationally dominated by a third one, by the work of Abramovich -
Karu - Matsuki - Wlodarczyk; see Matsuki \cite{Mk} (5-2-1); see also Lemma \ref{quot}.

\par
(6) We use $g | Y$ to signify that $g \in \Aut(Y)$.

\par
(7) In this paper, by a {\it pair} $(Y, g)$ we mean a projective
variety $Y$ and an automorphism $g \in \Aut(Y)$. Two pairs $(Y', g)$
and $(Y'', g)$ are $g$-{\it equivariantly birational}, if there is
a birational map $\sigma : Y' \cdots \to Y''$ such that
the birational action $\sigma (g|Y') \sigma^{-1} : Y'' \cdots \to Y''$
extends to a biregular action $g | Y''$.

\par
(8) $g \in \Aut(Y)$ is {\it periodic} if the order $\ord(g)$ is
finite. $g$ is {\it parabolic} if $\ord(g) = \infty$ and $g$ is
of null entropy.

\par
(9) $(Y', g)$ is {\it rigidly parabolic} if ($g | Y'$ is parabolic
and) for every pair $(Y, g)$ which is $g$-equivariantly birational
to $(Y', g)$ and for every $g$-equivariant surjective morphism $Y
\rightarrow Z$ with $\dim Z > 0$, we have $g | Z$ parabolic.

\par
(10) Let $Y'$ be a projective variety and $g \in \Aut(Y')$ of
positive entropy (so $\dim Y' \ge 2$). A pair $(Y', g)$ is of {\it
primitively positive entropy} if it is not of imprimitive positive
entropy, while a pair $(Y', g)$ is of {\it imprimitively positive
entropy} if it is $g$-equivariantly birational to a pair $(Y, g)$
and if there is a $g$-equivariant surjective morphism $f: Y
\rightarrow Z$ such that either one of the two cases below occurs.

\begin{itemize}
\item[(10a)]
$0 < \dim Z < \dim Y$, and $g | Z$ is still of positive entropy.
\item[(10b)]
$0 < \dim Z < \dim Y$, and $g | Z$ is periodic.
\end{itemize}

\par
(11) {\bf Remark.} We observe that in Case(10b), for some $s
> 0$ we have $g^s | Z = \id$ and that $g^s$ acts faithfully on the
general fibre $Y_z$ of $Y \rightarrow Z$, such that $g^s | Y_z$ is
of positive entropy. To see it, we replace $g^s$ by $g$ for
simplicity. In view of Lemma \ref{quot}, we may assume that $Y_z$ is
connected by making use of the Stein factorization, and also that
both $Y$ and $Z$ are smooth after $g$-equivariant resolutions as in
Hironaka \cite{Hi}. Let $0 \ne v_g \in \Nef(Y)$ be a nef divisor as
in Lemma \ref{PF} such that $g^* v_g = d_1(g) v_g$ with $d_1(g) : = d_1(g|Y) >
1$.
\par
We claim that $v_g | Y_z \ne 0$. Indeed, take very ample divisors
$H_1, \dots$, $H_k$ with $k = \dim Z$. Then $f^*H_1 \dots f^*H_k = c
Y_z$ homologously  with $c = (H_1 \dots H_k) > 0$. Noting that
$f^*H_1 \dots f^*H_i \ne 0$ and $g^*(f^*H_j) = f^*(g^*H_j) = f^*H_j$
and applying Lemma \ref{DS} repeatedly, we get $f^*H_1$ $\dots
f^*H_i . v_g \ne 0$ for all $i \le k$. In particular, $0 \ne f^*H_1
\dots f^*H_k . v_g = c Y_z . v_g = c v_g | Y_z$ homologously;
see Fulton \cite{Fu} (8.3) for the last equality.
This claim is proved.
\par
Next we claim that $d_{k+1}(g) := d_{k+1}(g|Y) \ge d_1(g | Y_z) \ge d_1(g)$
with $k = \dim Z$, so that $g | Y_z$ is of positive entropy. Indeed,
$g^*(v_g | Y_z) = d_1(g) (v_g | Y_z)$ implies that $d_1(g | Y_z) \ge
d_1(g)$. 
By Lemma \ref{PF}, $g^*
(v_{(g|Y_z)}) = (g | Y_z)^* v_{(g | Y_z)} = d_1(g | Y_z) v_{(g |
Y_z)}$, so $d_{k+1}(g) \ge d_1(g | Y_z)$.
\par
In \cite[Appendix, Theorem D]{NZ},
we will show that $d_1(g^s|Y) = d_1(g^s|Y_z)$
for the $s > 0$ as given at the beginning of this Remark.

\par \vskip 1pc

(12) {\bf Remark.} By the observation above and Lemma
\ref{Ddeg}, if $\dim Y \le 2$ and if the pair $(Y, g)$ is of
positive entropy, then $\dim Y = 2$ and the pair $(Y, g)$ is always
of primitively positive entropy.

\par \vskip 1pc
(13) We refer to Iitaka \cite{Ii82}  for the definition of $D$-{\it
dimension} $\kappa(X, D)$; the {\it Kodaira dimension} $\kappa(X) =
\kappa(\tilde{X}) = \kappa(\tilde{X}, K_{\tilde X})$ with $\tilde{X}
\rightarrow X$ a projective resolution; and the {\it Iitake fibring}
(of $X$): $X' \rightarrow Y'$ with $X'$ birational to $X$, both $X'$
and $Y'$ smooth projective, $\dim Y' = \kappa(X)$ ($= \kappa(X')$)
and $\kappa(X_y') = 0$ for a general fibre $X_y'$ over $Y'$. Note
that $\kappa(X)$ attains one of the values: $-\infty$, $0, 1, \,
\dots, \dim X$. We say that $X$ is of {\it general type} if
$\kappa(X) = \dim X$.

\par \vskip 1pc
(14) {\bf Remark.} The Iitaka fibring is defined by the
pluri-canonical system $|rK_X|$ for $r >> 0$ after $g$-equivariant
blowup to resolve base points in the system; see Hironaka \cite{Hi}.
So we can always replace $(X, g)$ by its $g$-equivariant blowup
$(X', g)$ such that there is a $g$-equivariant Iitaka fibring $X'
\rightarrow Y'$ with projective manifolds $X'$ and $Y'$ and with $\dim
Y' = \kappa(X') = \kappa(X)$. Note that $\kappa(X)$ is a birational
invariant.

\par \vskip 1pc
(15) A projective manifold $X$ of dimension $n$ is {\it
uniruled} if there is a dominant rational map $\P^1 \times Y \cdots
\rightarrow X$, where $Y$ is a projective manifold of dimension
$n-1$.

\par
(16) A projective manifold $X$ is a $Q$-{\it torus} in the sense of
\cite{Na99} if
there is a finite \'etale cover $T \to X$ from a torus $T$.
\par
(17) A projective manifold is a {\it weak Calabi-Yau manifold}
(or {\it wCY} for short) if the Kodaira dimension $\kappa(X) = 0$ and
if the {\it irregularity} $q(X) = h^1(X, \OO_X) = 0$.
A normal projective variety $\bar{X}$ with only terminal singularity
is a {\it Calabi-Yau variety} if the canonical divisor $K_{\bar X}$
satisfies $s K_{\bar X} \sim 0$ for some integer $s > 0$
and if $q({\bar X}) = 0$. So a projective resolution $X$ of
a Calabi-Yau variety $\bar{X}$ is a weak Calabi-Yau manifold.
Conversely, assuming the minimal model program,
every weak Calabi-Yau manifold is birational to
a Calabi-Yau variety.
We refer to \cite{KMM} or \cite{KM} for the definition of
singularity of type: {\it terminal, canonical, log terminal}, {\it or rational}.
\par
(18) An algebraic integer $\lambda > 1$ of degree $2(r+1)$ over $\Q$
with $r \ge 0$, is a {\it Salem} number (see \cite{Sa} or
\cite{McK3} \S 3) if all conjugates of $\lambda$ over $\Q$ (including $\lambda$ itself) are
given as follows, where $|\alpha_i| = 1$:
$$\lambda, \, \lambda^{-1}, \, \alpha_1, \, \bar{\alpha}_1, \,
\dots, \, \alpha_r, \, \bar{\alpha}_r.$$

\par \vskip 1pc
The following result is fundamentally important in the study of
complex dynamics. For the proof,
we refer the readers to \cite{Gr}, \cite{Yo}, \cite{Fr95},
\cite{DS05} (2.5) and page 302,
\cite{Di} Proposition 5.7,
\cite{DS} before (1.4), \cite{Fr06} the Introduction,
\cite{Gu03} (1.2), (1.5), (1.6).

\begin{lemma}\label{Ddeg}
Let $X$ be a projective manifold of dimension $n$,
and $g \in \Aut(X)$.
Then the following are true.
\begin{itemize}
\item[(1)]
$d_i(g) = \rho(g^* | H^{i,i}(X, \R)) = \rho(g^* | H_a^{i,i}(X, \R))$,
$1 \le i \le n$.
\item[(2)]
$d_1(g) = \rho(g^* | \NS_{\R}(X))$ $($see also \cite{DS05} $(3.5))$.
\item[(3)]
$h(g) = \log \rho(g) = \max_{1 \le i \le n} \log d_i(g)$.
\item[(4)]
The entropy $h(g) > 0$ holds if and only if the dynamical degree
$d_{\ell}(g)
> 1$ for one $($and hence for all$)$ $1 \le \ell \le n-1$ $($by
$(5)$$)$.
\item[(5)]
The map $\ell \mapsto d_{\ell}(g)/d_{\ell + 1}(g)$ with $0 \le \ell
\le n-1$, is non-decreasing. So $d_{\ell}(g) \le d_1(g)^{\ell}$ and
$d_{n - \ell}(g) \le d_{n-1}^{\ell}(g)$ for all $1 \le \ell \le n$.
Also there are integers $m, m'$ such that:
$$1 = d_0(g) < d_1(g) < \cdots < d_m(g) = \cdots = d_{m'}(g) > \cdots > d_n(g) = 1.$$
\end{itemize}
\end{lemma}

The following very useful result is proved in Dinh-Sibony
\cite{DS04} (3.2), (4.4).

\begin{lemma}\label{DS}
Let $X$ be a projective  manifold of dimension $n$.
Let $\Nef(X) \ni P$, $P'$, $P_i$ $(1 \le i \le m; \, m \le n-2)$ be nef divisors.
Then we have:
\begin{itemize}
\item[(1)]
Suppose that $P_1 . P_2 = 0$ in $H^{2,2}(X, \R)$. Then one of $P_1,
P_2$ is a multiple of the other.
\item[(2)]
We have $P_1 \cdots P_m . P . P' \ne 0 \in H^{m+2, m+2}(X, \R)$ if
the two conditions below are satisfied.

\begin{itemize}
\item[(2a)]
$P_1 \dots P_m . P \ne 0$ and $P_1 \dots P_m . P' \ne 0$.
\item[(2b)]
One has $g^*(P_1 \dots P_m . P) = \lambda (P_1 \dots P_m . P)$ and
$g^*(P_1 \dots P_m . P') = \lambda' (P_1 \dots P_m . P)$, for some
$g \in \Aut(X)$ and distinct (positive) real numbers $\lambda$ and
$\lambda'$.
\end{itemize}

\end{itemize}
\end{lemma}

We refer to Dinh-Sibony \cite{DS05} (3.5) for a result including
the one below and with more analytical information.

\begin{lemma} (Generalized Perron-Frobenius Theorem) \label{PF}
Let $X$ be a projective manifold and $g \in \Aut(X)$. Then
there are non-zero nef divisors $v_g$ and $v_{g^{-1}}$ in $\Nef(X)$
such that:
$$g^* v_g = d_1(g) v_g, \hskip 1pc
(g^{-1})^* v_{g^{-1}} = d_1(g^{-1}) v_{g^{-1}}.$$
\end{lemma}

\begin{proof}
To get the first equality, we apply to the nef cone $\Nef(X)$ of the
Perron - Fobenius Theorem for cones as in Schneider - Tam \cite{ST}
page 4, item 5. The second is the
application of the first to $g^{-1}$. This proves the lemma.
\end{proof}

Here is the relation between dynamical degrees of automorphisms.

\begin{lemma}\label{hii}
Let $X$ be a projective  manifold of dimension $n$,
and $g \in \Aut(X)$. Then we have:
\begin{itemize}
\item[(1)]
Denote by $\Sigma(g)_i = \Sigma(g | X)_i$ the set of all eigenvalues of $g^* |
H^{i,i}(X, \C)$ (including multiplicities). Then $\Sigma(g)_1 =
\Sigma(g^{-1})_{n-1}$.
\item[(2)]
The dynamical degrees satisfy $d_1(g) = d_{n-1}(g^{-1})$.
\item[(3)]
$g$ is of positive entropy (resp. null entropy; periodic; parabolic)
if and only if so is $g^{m}$ for some (and hence for all) $m \ne 0$.
\end{itemize}
\end{lemma}

\begin{proof}
There is a natural perfect pairing
$$H^{1,1}(X, \C) \times H^{n-1,n-1}(X, \C) \rightarrow \C$$
induced by the cup product, via the Hodge
decomposition. This pairing is preserved by the action of $g^*$; see
Griffiths - Harris \cite{GH} page 59. So a simple linear algebraic
calculation shows that if $g^* | H^{1,1}(X, \C)$ is represented by a
matrix $A$ then $g^* | H^{n-1,n-1}(X, \C)$ is represented by the
matrix $(A^t)^{-1}$. Thus the lemma follows; see Lemma \ref{Ddeg}.
\end{proof}

The result below shows that the first dynamical degree of an automorphism
is preserved even after lifting up or down by a generically
finite and surjective morphism.

\begin{lemma}\label{quot}
Let $f: X \rightarrow Y$ be a $g$-equivariant generically finite
surjective morphism between projective manifolds of dimension $n \ge 2$.
Then we have:
\begin{itemize}
\item[(1)]
$d_1(g | X) = d_1(g | Y)$.
\item[(2)]
$g | X$ is of positive entropy (resp. null entropy;  periodic) if and only if so is $g | Y$.
\item[(3)]
$g | X$ is of parabolic if and only if so is $g | Y$.
\item[(4)]
If $g | X$ is rigidly parabolic then so is $g | Y$.
\end{itemize}
\end{lemma}

\begin{proof}
(1) Set $\Sigma_X = \Sigma(g^* | \NS_{\R}(X)) = \{\lambda \in \R \ |
g^*D = \lambda D$ for a divisor $0 \ne D \in \NS_{\R}(X)\}$. We show
first that $\Sigma_Y \subseteq \Sigma_X$, whence $d_1(g|Y) \le
d_1(g|X)$ by Lemma \ref{Ddeg}. Indeed, assume that $g^*\bar{L} =
\lambda \bar{L}$ for some $0 \ne \bar{L} \in \NS_{\R}(Y)$ and
$\lambda \in \Sigma_Y$. Set $L := f^*\bar{L}$. Then $g^*L = f^* g^*
\bar{L} = \lambda L$. Note that $0 \ne L \in \NS_{\R}(X) \subset
H^2(X, \R)$ because $f^* : H^*(Y, \R) \rightarrow H^*(X, \R)$ is an
injective ring homomorphism; see \cite{BHPV} I (1.3). Therefore,
$\lambda \in \Sigma_X$.
\par
Conversely, let $0 \ne L: = v_g \in \Nef(X)$ such that $g^*L = d_1
L$ with $d_1 = d_1(g | X)$, as in Lemma \ref{PF}. Set $\bar{L} :=
f_*L$. For any $H \in H^{2n-2}(Y, \R)$, by the projection formula, we
have $g^*\bar{L} . H = \bar{L} . g^{(-1)*} H = L . f^* g^{(-1)*} H =
L . g^{(-1)*} f^* H = g^*L . f^*H = d_1 L . f^*H = d_1 \bar{L} . H$.
So $(g^*\bar{L} - d_1 \bar{L}) . H = 0$ for {\it all} $H \in
H^{2n-2}(Y, \R)$. Hence $g^*\bar{L} = d_1 \bar{L}$ in $H^2(Y, \R)$.
\par
We claim that $\bar{L} \ne 0$ in $\NS_{\R}(Y)$, whence $d_1 = d_1(g
| X) \in \Sigma_Y$, $d_1(g | X) \le d_1(g | Y)$ by Lemma
\ref{Ddeg}, and we conclude the assertion (1). Assume the contrary
that $\bar{L} = 0$. Take an ample divisor $H_Y$ on $Y$. Then
$f^*H_Y$ is nef and big on $X$. So $f^*H_Y = A + D$ for an ample
$\Q$-divisor $A$ and an effective $\Q$-divisor on $X$, by Kodaira's
lemma. By the projection formula and the nefness of $L$, one has $0
= \bar{L} . H_Y^{n-1} = L . f^*H_Y^{n-1} = L . f^*H_Y^{n-2} . (A +
D) \ge L . f^*H_Y^{n-2} . A \ge \cdots \ge L . A^{n-1} \ge 0$.
Applying the Lefschetz hyperplane section inductively to reduce to
the Hodge index theorem for surfaces and using the nefness of $L$,
we see that $L = 0 \in \NS_{\R}(X) \subseteq H^2(X, \R)$, a
contradiction. So the claim and hence the assertion (1) are proved.
\par
With (1), the assertion (2) follows. Now (3) follows from (1) and
(2).
\par
(4) Assume that $g | X$ is rigidly parabolic. 
Modulo $g$-equivariant
birational modification, we have only to show that $g | Y_1$ is parabolic whenever
$Y \rightarrow Y_1$ is a $g$-equivariant surjective morphism with $\dim Y_1 > 0$.
This follows from the
assumption on $g | X$ and the $g$-equivariance of the composition $X
\rightarrow Y \rightarrow Y_1$. This proves the lemma.
\end{proof}

We now describe the behavior of automorphisms dynamics in a fibration.

\begin{lemma}\label{fib}
Let $X \rightarrow Y$ be a $g$-equivariant surjective morphism
between projective manifolds with $\dim X > \dim Y > 0$. Then we have:
\begin{itemize}
\item[(1)]
If $g | X$ is of null entropy (resp. periodic), then so is $g | Y$.
\item[(2)]
Suppose that the pair $(X, g)$ is either rigidly parabolic or of
primitively positive entropy. Then $g | Y$ is rigidly parabolic.
\end{itemize}
\end{lemma}

\begin{proof}
The proof of (1) is similar to that of Lemma \ref{quot}. Suppose the
contrary that (2) is false for some $Y$ in (2). 
Then, after $g$-equivariant birational modification, 
there is a
$g$-equivariant surjective morphism $Y \rightarrow Z$ with $\dim Z >
0$ such that $g | Z$ is not parabolic. Thus, $g | Z$ is periodic or
of positive entropy. This contradicts the rigidity or primitivity of
$(X, g)$ because $0 < \dim Z \le \dim Y < \dim X$.
\end{proof}

Here is the relation between Salem numbers and dynamical degrees in
\ref{entropy}:

\begin{lemma}\label{salem}
Let $X$ be a projective manifold and $g \in \Aut(X)$ of positive entropy.
Then we have:
\begin{itemize}
\item[(1)]
If $\dim X = 2$, then $d_1(g) = d_1(g^{-1}) = e^{h(g)}$ and it is a Salem number.
\item[(2)]
Suppose $\dim X = 3$ and there is a $g$-equivariant morphism $f : X
\rightarrow Y$ onto a smooth projective curve $Y$ with connected
general fibre $F$. Then all $e^{h(g^{\pm})}$, $d_1(g^{\pm})$,
$d_2(g^{\pm})$ are equal and it is a Salem number.
\end{itemize}
\end{lemma}

\begin{proof}
The result in Case(1) is well known and follows from Lemmas
\ref{Ddeg} and \ref{hii} and the proof of McMullen \cite{McK3}
Theorem 3.2.
\par
We consider Case(2).
Set $L : = (\NS(X) | F)/(\torsion) \subset \NS(F)/(\torsion)$. We define the following intersection form $\langle , \, \rangle_L$
on the lattice $L$:
$$\langle D_1 | F, D_2 | F \rangle_L := D_1 . D_2 . F \in H^6(X, \Z) \cong \Z.$$
This $\langle , \, \rangle_L$ is compatible with the intersection
form on $\NS(F)$ via the restriction $H^2(X, \Z) \rightarrow H^2(F,
\Z)$.
 This compatibility, the Hodge index theorem for the smooth
projective surface $F$, and the fact that $H | F \ne 0$ in $L$ with $H$ an
ample divisor of $X$, imply that the lattice $L$ is non-degenerate
and has signature $(1, r)$ with $1 + r \le \rank \, \NS(F)$. There
is a natural action $g^* | L$ on $L$ given by $g^*(D|F) = (g^* D) |
F$. Since $g^*F = F$ in $\NS(X)$, this action is well defined and
preserves the intersection form $\langle , \, \rangle_L$.
\par
Since $g^* F = F$ and $g^* v_g = d_1 v_g$ with $d_1 = d_1(g)
> 1$ in the notation of Lemma
\ref{PF}, our $F$ and $v_g$ are not proportional. So the Lefschetz
hyperplane section theorem on cohomology and the Cauchy-Schwarz
inequality or the Hodge index theorem for surfaces imply that $v_g .
F . A = (v_g |A) . (F | A) \ne 0$ for a very ample divisor $A$ on
$X$. So $v_g . F = v_g | F$ gives a non-zero $v := v_g | F \in L \otimes_{\bold Z} {\bold R}$.
Further, $g^* v = (g^*v_g) | F = d_1 v$. Since $L$ is an integral
lattice of signature $(1, r)$ and $g^* | L$ is an isometry of $L$,
by the proof of McMullen \cite{McK3} Theorem 3.2, $d_1$ is a Salem
number and all eigenvalues of $g^* | L$ are given as:
$$d_1, \, d_1^{-1}, \, \alpha_1, \, \bar{\alpha}_1, \, \dots,
\alpha_t, \, \bar{\alpha}_t$$ with $|\alpha_i| = 1$ and $2(t+1) = r
+ 1$. Arguing with $g^{-1} | L$, we get $d_1(g | X) = d_1(g^{-1} |
X)$ ($= d_2(g|X)$ by (\ref{hii})). The result follows; see Lemma
\ref{Ddeg}.
\end{proof}

The following result (though it will not be used in the sequel) is a
generalization of a well-known result in the case of surfaces.

\begin{lemma}\label{nefs}
Let $X$ be a projective manifold of dimension $n \ge 2$ and
$g \in \Aut(X)$ of positive entropy. Let $0 \ne v_i \in \Nef(X)$ $(1
\le i \le s)$ be nef divisors such that $g^* v_i = \lambda_i v_i$
for scalars $\lambda_i$ with $\lambda_i > 1$ and that $\lambda_i$
are pairwise distinct. Then we have:
\begin{itemize}
\item[(1)]
$s \le n-1$.
\item[(2)]
If $s = n-1$, then $d_1(g) = \max_{1 \le i \le s} \{\lambda_i\}$.
\end{itemize}
\end{lemma}

\begin{proof}
(1) Applying Lemma \ref{DS} repeatedly, we see that $u(s_1):=
\prod_{i=1}^{s_1} v_i \ne 0$ if $s_1 \le n$.
Note that $g^* = \id$ on $H^{n,n}(X, \R) \cong \R$.
If $s \ge n$, then
$u(n)$ is a non-zero scalar in $H^{n,n}(X, \R)$, whence $u(n) = g^*
u(n) = \lambda u(n)$ with $\lambda := \prod_{i=1}^{n} \lambda_i
> 1$. This is a contradiction.
\par
(2) Assume that $s = n-1$. If $d_1 := d_1(g)$ is one of $\lambda_i$,
then (2) is true by the maximality of $d_1(g)$ as in Lemma
\ref{Ddeg}. Suppose that $d_1 \ne \lambda_i$ for all $i$. one gets a
contradiction to (1) if one sets $v_{n} = v_g$ in the notation of
Lemma \ref{PF}. The lemma is proved.
\end{proof}

The result below shows that one may tell about the im/primitivity of
threefold automorphisms by looking at the algebraic property of its
dynamical degrees or entropy.

\begin{lemma}\label{3nef}
Let $X$ be a smooth projective threefold and $g \in
\Aut(X)$ such that the pair $(X, g)$ is of
imprimitively positive entropy. Then we have:
\begin{itemize}
\item[(1)]
All scalars $e^{h(g^{\pm})}$, $d_1(g^{\pm})$, $d_2(g^{\pm})$
are equal and it is a Salem number.
\item[(2)]
For some $s > 0$, we have
$d_1(g^s) \in D_1^t(2) \cup D_1^{cy}(2) \cup D_1^{rc}(2)$.
\end{itemize}
\end{lemma}

We now prove Lemma \ref{3nef}. After $g$-equivariant birational modification, 
there is a
$g$-equivariant surjective morphism $f: X \rightarrow Y_1$ such that
$\dim X > \dim Y_1 > 0$ and that either $g | Y_1$ is of positive entropy
or $g | Y_1$ is periodic.
Let $X \rightarrow Y \rightarrow Y_1$ be
the Stein factorization.
After $g$-equivariant blowups as in Hironaka \cite{Hi},
we may assume that $X$, $Y$ and $Y_1$
are all smooth, $X \rightarrow Y$ has connected fibres $F$
and $Y \rightarrow Y_1$ is generically finite and surjective.
By Lemma \ref{quot},
either $g | Y$ is of positive entropy or $g | Y$ is periodic.
Since Question \ref{question} has a positive answer in dimension 2
as in Cantat \cite{Ca99}, our
lemma follows from Lemma \ref{salem} and the claim below.

\begin{claim} We have:
\begin{itemize}
\item[(1)]
Suppose that $g | Y$ is of positive entropy. Then the four scalars
$d_1(g^{\pm} | X)$, $d_1(g^{\pm} | Y)$ coincide and we denote it by
$d_1$.
\item[(2)]
Suppose that $g | Y$ is periodic say $g^s | Y = \id$ for some $s >
0$. Then the four scalars $d_1(g^{\pm s} | X)$, $d_1(g^{\pm s} | F)$
coincide and we denote it by $d_1^s$ or $d_1(g^s)$. So $d_1(g|X) =
d_1(g^{-1} | X) = d_1$.
\item[(3)]
For both cases in (1) and (2), if $0 \ne P \in \Nef(X)$ is a nef divisor such
that $g^*P = \lambda P$ then $\lambda \in \{1, d_1^{\pm}\}$.
\item[(4)]
For both cases in (1) and (2), $d_1 = d_1(g^{\pm}|X) = d_2(g^{\pm}|X) = e^{h(g^{\pm}|X)}$.
\end{itemize}
\end{claim}

Let us prove the claim.
(4) follows from (1) - (2) and Lemmas \ref{hii} and \ref{Ddeg}.
Consider the case where $g | Y$ is of
positive entropy. Thus $\dim Y \ge 2$, whence $\dim Y = 2$ and the
fibre $F$ is of dimension $1$. The two scalars
$d_1(g^{\pm} | Y)$ coincide and we denote it by $d_1$; see Lemma
\ref{hii}. Set $L^{\pm} := f^*v_{(g^{\pm} | Y)} \in \Nef(X)$;
see Lemma \ref{PF} for the notation.
Note that $L^{\pm} \ne 0$ in $\NS_{\R}(X)$;
see \cite{BHPV} I (1.3).
Further, $g^* L^{\pm} =
f^* g^* v_{(g^{\pm} | Y)} = d_1^{\pm} L^{\pm}$.
Thus, to prove the claim in the present case,
we only have to show (3); see Lemma \ref{PF}.
If $\lambda \ne d_1^{\pm}$, then $u: = L^+ . L^- . P \in H^{3,3}(X, \R) \cong \R$
is a non-zero scalar by Lemma \ref{DS} and
hence $u = g^* u = d_1 d_1^{-1} \lambda u = \lambda u$. So $\lambda = 1$.
This proves the claim for the present case.

\par
Consider the case where $g^s | Y = \id$. As remarked in
(\ref{entropy}), $g^s | F$ is of positive entropy, so $\dim F = 2$
and $\dim Y = 1$. Further, $d_1(g^{-s} | X) = d_2(g^s | X) \ge
d_1(g^s | F) \ge d_1(g^s | X)$; see also Lemma \ref{hii}. Arguing
with $g^{-1}$, we get $d_1(g^s | X) \ge d_1(g^{-s} | F) \ge
d_1(g^{-s} | X)$. Since the two scalars $d_1(g^{\pm s} | F)$
coincide for surface $F$ by Lemma \ref{hii}, the above two sequences
of inequalities imply (2). Consider $u = v_{(g | X)} . v_{(g^{-1} |
X)} . P$ as in the early case, one proves (3) for the present case.
This proves the claim and also the lemma.

\par \vskip 1pc
The following lemma is crucial, which is derived from a result of
Nakamura - Ueno, and Deligne as in Ueno \cite{Ue}
Theorem 14.10. See \cite[Theorem A]{NZ} for the generalization
of \cite{Ue} Theorem 14.10 to meromorphic dominant maps on K\"ahler manifolds.

\begin{lemma}\label{deligne}
Let $X \rightarrow Y$ a $g$-equivariant surjective morphism between
projective manifolds with $\dim Y > 0$. Suppose that the pair $(X,
g)$ is either rigidly parabolic or of primitively positive entropy.
Then the Kodaira dimension $\kappa(Y) \le 0$. In
particular, $\kappa(X) \le 0$.
\end{lemma}

\begin{proof}
If $\kappa(Y) = \dim Y$, then $\Aut(X)$ is a finite group; see
Iitaka \cite{Ii82} Theorem 11.12. Thus $g | Y$ is periodic, which is impossibe
by our assumption on $g$.
Suppose that $0 < \kappa(Y) < \dim Y$. After replacing with
$g$-equivariant blowups of $X$ and $Y$ as in Hironaka \cite{Hi}, we
may assume that $Y \rightarrow Z$ is a well-defined Iitaka fibring
with $Y$ and $Z$ smooth projective and $\dim Z = \kappa(Y) > 0$. The
natural homomorphism $\Bir(Y) \rightarrow \Bir(Z)$ between
birational automorphism groups, has a finite group as its image; see
\cite{Ue} Theorem 14.10. In particular, $g | Z$ is periodic. This
contradicts the assumption on $g$, noting that the composition $X
\rightarrow Y \rightarrow Z$ is $g$-equivariant. This proves the
lemma.
\end{proof}

For $g$ of positive entropy (not necessarily being primitive),
we have:

\begin{lemma} \label{k-2}
The following are true.
\begin{itemize}
\item[(1)]
Let $X$ be a projective manifold of dimension $n$. Suppose that $g \in \Aut(X)$ is of
positive entropy. Then the Kodaira dimension $\kappa(X) \le n - 2$.
\item[(2)]
Conversely, for every $n \ge 2$ and every $k \in \{-\infty, \ 0, \ 1, \ \dots, \ n-2\}$,
there are a projective manifold $X$ and $g \in \Aut(X)$ of positive
entropy such that $\dim X = n$ and $\kappa(X) = k$.
\end{itemize}
\end{lemma}

\begin{proof}
(1) Assume the contrary that $\kappa(X) \ge n-1$. After
$g$-equivariant blowup as in Hironaka \cite{Hi}, we may assume that
for some $s > 0$, one has $|sK_X| = |M| + Fix$ with $Fix$ the fixed
part and with the movable part $|M|$ base point free, so that $f : =
\Phi_{|M|} : X \rightarrow Y \subset \P^N$ with $N = h^0(X, M)$, is
the ($g$-equivariant) Iitaka fibring. Note that $M^{\kappa(X)}$ is
homologous to a positive multiple of a fibre of $f$ and hence $M^r
\ne 0$ for every $r \le \kappa(X)$. Also $g^*M \sim M$ (linearly
equivalent). With the notation of Lemma \ref{PF} and by Lemma
\ref{DS}, we have $v_g . M^r \ne 0$ in $H^{r+1,r+1}(X, \R)$ for all
$r \le n-1$. Since $g^* = \id$ on $H^{n,n}(X, \R) \cong \R$, for the
scalar $u := v_g . M^{n-1} \in H^{n,n}(X, \R)$, we have $u = g^* u =
d_1 u$ with $d_1 = d_1(g) > 1$, so $u = 0$, a contradiction.
\par
(2) Let $S$ be a surface with $g \in \Aut(S)$ of positive entropy.
Let $Z$ be any $(n-2)$-fold. Set $X := S \times Z$ and $g | X :=
(g|S) \times (\id_Z)$. Then $g |X$ is of positive entropy by looking
at the K\"unneth formula for $H^2(X, \C)$ as in Griffiths-Harris
\cite{GH} page 58; see also \cite{Di} Proposition 5.7. 
Also $\kappa(X) = \kappa(S) + \kappa(Z)$. All
values in $\{-\infty, \ 0\}$ (resp. $\{-\infty, \ 0, \ 1, \ \dots, \
n-2\}$) are attainable as the Kodaira dimension of a suitable $S$ (resp. $Z$);
see, for instance, Cantat \cite{Ca01} and McMullen \cite{McP2}. Thus
(2) follows. This proves the lemma.
\end{proof}

We need the following result on the eigenvalues of $g^* | H^*(X, \C)$.

\begin{lemma}\label{uni}
Let $X$ be a projective manifold of dimension $n$, and $g \in \Aut(X)$
of null entropy. Then there is an integer $s > 0$ such that $(g^s)^*
| H^*(X, \C)$ is unipotent, i.e., all eigenvalues are equal to $1$.
\end{lemma}

\begin{proof}
By Lemmas \ref{Ddeg} and \ref{hii}, every eigenvalue $\lambda$ of
$g^* | H^*(X, \C)$ has modulus $1$. Since $g^*$ is defined over
$\oplus_{i=0}^{2n} H^i(X, \Z)/(\torsion)$ the monic characteristic
polynomial of $g^* | H^*(X, \C)$ has integer coefficients,
whence all eigenvalues $\lambda$ of $g^*$ are algebraic integers. So
every eigenvalue $\lambda$ of $g^*$ is an algebraic integer and all
its conjugates (including itself) have modulus $1$. Thus these
$\lambda$ are all units of $1$ by Kronecker's theorem. The lemma follows.
\end{proof}

The result below says that a rigidly parabolic action on an abelian
variety is essentially the lifting of a translation.

\begin{lemma}\label{Abel}
Let $A \ne 0$ be an abelian variety and $g \in \Aut_{\rm variety}(A)$.
Suppose that the pair $(A, g)$ is rigidly parabolic.
Then there are integers $s > 0$,
$m \ge 0$ and a sequence of abelian subvarieties $B_1 \subset B_2
\dots \subset B_m \subset A$ such that the following are true
(setting $B_0 = 0$).
\begin{itemize}
\item[(1)]
The homomorphisms below are
all $g^s$-equivariant:
$$A \rightarrow A/B_1 \rightarrow \cdots \rightarrow A/B_m \ne 0.$$
\item[(2)]
$g^s | (A/B_m)$ is a translation of infinite order
(so the pointwise fixed point set $(A/B_i)^{g^{r}} = \emptyset$
for all $r \in s \N$ and all $0 \le i \le m$).
\end{itemize}
\end{lemma}

\begin{proof}
We may assume that $g^* | H^*(A, \C)$ is already unipotent; see
Lemma \ref{uni}. Assume that the pointwise fixed locus $A^g \ne
\emptyset$. Then we may assume that $g | A$ is a homomorphism after
changing the origin. By \cite{BL} (13.1.2), $A^g$ is a subgroup
of positive dimension equal to that of the eigenspace of $g^* |
H^{1,0}(A, \C)$. Let $B_1$ be the identity component of $A^g$. Then
$0 < \dim B_1 < \dim A$ by the parabolic rigidity of $g | A$. The
homomorphism $A \rightarrow A/B_1$ is $g$-equivariant. But now
$(A/B_1)^g$ contains the origin ($=$ the image of $B_1$), and $g |
(A/B_1)$ is again rigidly parabolic by the definition. The parabolic
rigidity of $g | A$ helps us to continue this process forever. This
contradicts the finiteness of $\dim A$.
\par
Therefore, $A^g = \emptyset$. Write $g = t_0 h$ with a translation
$t_0$ and a homomorphism $h$. Then $0 = |A^g| = |A^h|$ in the
notation of \cite{BL} (13.1.1), whence $A^h$ is of positive
dimension equal to that of $\Ker(h^* - \id) = \Ker(g^* - \id)
\subset H^{1,0}(A, \C)$. Let $B_1$ be the identity component of
$A^h$. Then $g(x + B_1) = g(x) + B_1$ and hence the homomorphism $A
\rightarrow A/B_1$ is $g$-equivariant. If $B_1 = A$, then $h = \id$
and we are done. Otherwise, $g | (A/B_1)$ is again rigidly parabolic
by the definition. Also $(A/B_1)^{g^r} = \emptyset$ for all $r > 0$
by the argument in the paragraph above for the early case. Continue
this process and we see that $B_{m+1} = A$ for some $m$. So some
positive power $g^s | (A/B_m)$ is a translation of infinte order.
The lemma is proved.
\end{proof}

The density result (3) below shows that
a rigidly parabolic action on an abelian variety is very ergodic.

\begin{lemma}\label{smooth}
Let $A \rightarrow Y$ be a $g$-equivariant generically finite
surjective morphism from an abelian variety $A$ onto a
projective manifold $Y$. Assume that $g |
A$ is rigidly parabolic. Then we have:
\begin{itemize}
\item[(1)]
No proper subvariety of $Y$ is stabilized by a positive power of
$g$.
\item[(2)]
$g$ has no periodic points.
\item[(3)]
For every $y_0 \in Y$, the Zariski closure $D :=
\overline{\{g^s(y_0) \, | \, s > 0\}}$ equals $Y$.
\item[(4)]
Suppose that $f : X \rightarrow Y$ is a $g$-equivariant surjective
morphism from a projective  manifold onto $Y$. Then $f$ is a
smooth morphism. In particular, if $f$ is generically finite then it
is \'etale.
\end{itemize}
\end{lemma}

\begin{proof}
It suffices to show (1). Indeed, (2) is a special case of (1). If
(3) is false, then some positive power $g^s$ fixes an irreducible
component of $D$, contradicting (1). If (4) is false, then the
discriminant $D = D(X/Y)$ is stabilized by $g$ and we get a
contradiction as in (3).
\par
For (1), since $A \rightarrow Y$ is $g$-equivariant, it is enough to
show (1) for $A$; see Lemma \ref{quot}. Suppose the contrary that a
positive power $g^{v}$ stabilizes a proper subvariety $Z$ of $A$. To
save the notation, rewrite $g^{v}$ as $g$. If $Z$ is a point, then
$g^s | (A/B_m)$ fixes the image on $A/B_m$ of $Z$ in the notation of
Lemma \ref{Abel}, absurd. Assume that $\dim Z > 0$. If the Kodaira
dimension $\kappa(Z) = 0$, then $Z$ is a translation of a subtorus
by Ueno \cite{Ue} Theorem 10.3, and we may assume that $Z$ is
already a torus after changing the origin; we let $Z$ be the $B_1$
in Lemma \ref{Abel} and then $g^s | (A/B_m)$ fixes the origin ($=$
the image of $Z$), absurd.
\par
Suppose that $\kappa(Z) \ge 1$. By Ueno \cite{Ue} or Mori \cite{Mo} (3.7),
the identity component $B_1$ of $\{a \in A \, | \, a + Z = Z\}$ has positive
dimension such that $Z \rightarrow Z/B_1 \subset A/B_1$ is
birational to the Iitaka fibring with $\dim (Z/B_1) = \kappa(Z)$ and
$Z/B_1$ of general type. We can check that $g(a + B_1) = g(a) +
B_1$ (write $g$ as the composition of a translation and a
homomorphism and then argue), so the homomorphism $A \rightarrow
A/B_1$ is $g$-equivariant and $g | (A/B_1)$ stabilizes a subvariety
$Z/B_1$ of general type (having finite $\Aut(Z/B_1)$ by Iitaka
\cite{Ii82} Theorem 11.12). Thus a positive power $g^{v} | (A/B_1)$
fixes every point in $Z/B_1$. So in Lemma \ref{Abel}, another positive power
$g^s | (A/B_m)$ fixes every
point in the image of $Z/B_1$, absurd.
\end{proof}

Here are two applications of Lemma \ref{deligne} and Viehweg-Zuo
\cite{VZ} (0.2).

\begin{lemma}\label{VZ}
Let $X$ be a projective  manifold of
Kodaira dimension $\kappa(X) \ge 0$ and $g \in \Aut(X)$. Suppose that $f : X \rightarrow
\P^1$ is a $g$-equivariant surjective morphism. Then $g | \P^1$ is
periodic. In particular, $(X, g)$ is neither rigidly parabolic nor of
primitively positive entropy.
\end{lemma}

\begin{proof}
By \cite{VZ} Theorem 0.2,
$f$ has at least three singular fibres lying
over a set of points of $\P^1$ on which $g | \P^1$
permutes. Thus a positive power $g^s | \P^1$ fixes every point
in this set and hence is equal to the identity.
\end{proof}

\begin{lemma}\label{n=1}
Let $f: X \rightarrow Y$ be a $g$-equivariant surjective morphism
from a projective manifold onto a smooth projective curve.
Suppose that the Kodaira dimension $\kappa(X) \ge 0$.
Suppose further that the pair $(X, g)$ is either rigidly parabolic,
or of primitively positive entropy $($so $\dim X \ge 2)$.
\par
Then $Y$ is an elliptic curve,
$g^s | Y$ (for some $s \, | \, 6$) is a translation of infinite order,
and $f$ is a smooth morphism.
\end{lemma}

\begin{proof}
By Lemma \ref{deligne}, the Kodaira dimension $\kappa(Y) \le 0$. So
the arithmetic genus $p_a(Y) \le 1$. By Lemma \ref{VZ}, $Y$ is an
elliptic curve. So $g^s | Y$ is a translation for some $s \, | \, 6$,
which is of infinite order since $g | Y$ is rigidly parabolic by
Lemma \ref{fib}. In view of Lemma \ref{smooth} the lemma follows.
\end{proof}

The following are sufficient conditions to have rational pencils on surfaces.

\begin{lemma}\label{k<1}
Let $X$ be a smooth projective surface of Kodaira dimension
$\kappa(X) \ge 0$, and let $g \in \Aut(X)$. Let $X \rightarrow X_m$
be the smooth blowdown to the minimal model. Then there is a unique
$g$-equivariant surjective morphism $\tau: X \rightarrow \P^1$ such that
$g | \P^1$ is periodic, if either Case $(1)$ or $(2)$ below occurs.
\begin{itemize}
\item[(1)]
$X_m$ is a hyperelliptic surface.
\item[(2)]
$g$ is parabolic. $X_m$ is $K3$ or Enriques.
\end{itemize}
\end{lemma}

\begin{proof}
Since $\kappa(X) \ge 0$ in both Cases $(1)$ and $(2)$,
the minimal model $X_m$ of $X$ is unique and hence
$X \rightarrow X_m$ is $g$-equivariant.
So we may assume that $X = X_m$.
\par
There are exactly two elliptic fibrations on a hyperelliptic surface
$X$ (see Friedman-Morgan \cite{FM} \S 1.1.4); one is onto an elliptic curve
($= \Alb(X)$) and the other is onto $\P^1$. Thus Lemma \ref{VZ}
implies the result in Case (1).
\par
If $X_m$ is $K3$,
then the existence of $\tau$ follows from Cantat \cite{Ca01} (1.4).
\par
An Enrques $X$ can be reduced to the $K3$ case. Indeed, $g | X$
lifts to a parabolic $g | Y$ (see Lemma \ref{quot}) on the universal
$K3$ double cover $Y$ of $X$ so that a positive power $g^s | Y$
stabilizes every fibre of an elliptic fibration on $Y$. This
fibration descends to one on $X$ fibre-wise stabilized by $g^s | X$.

\par
For the uniqueness of $\tau$ in Case (2), if $F_i$ are fibres of two
distinct $g$-equivariant fibrations, then $g^*$ stabilizes the class of the nef and big
divisor $F_1 + F_2$ and hence some positive power $g^r \in \Aut_0(X) = (1)$ (see Lemma \ref{nefbig}).
This is absurd.
This proves the lemma.
\end{proof}

Now we classify rigidly parabolic actions on surfaces.

\begin{lemma}\label{ps}
Let $X$ be a smooth projective surface and $g \in \Aut(X)$ such that
the pair $(X, g)$ is rigidly parabolic. Then there is a
$g^s$-equivariant (for some $s > 0$) smooth blowdown $X \rightarrow
X_m$ such that one of the following cases occurs (the description in
$(3)$ or $(4)$ will not be used in the sequel).
\begin{itemize}
\item[(1)]
$X_m$ is an abelian surface $($so $q(X) = 2)$; see also $(\ref{smooth})$.
\item[(2)]
$X_m \rightarrow E = \Alb(X_m)$ is an elliptic ruled surface (so
$q(X) = 1$ and $E$ an elliptic curve). $g^s | E$ is a translation of
infinite order.
\item[(3)]
$X_m$ is a rational surface such that $K_{X_m}^2 = 0$, the
anti-canonical divisor $-K_{X_m}$ is nef and the anti-Kodaira
dimension $\kappa(X_m, -K_{X_m}) = 0$. For a very general point $x_0
\in X_m$, the Zariski closure $D(x_0) : = \overline{\{g^r(x_0) \, |
\, r > 0\}}$ equals $X_m$.
\item[(4)]
$X_m$ is a rational surface with $K_{X_m}^2 = 0$
and equipped with a $($unique and relatively minimal$)$
elliptic fibration $f: X_m \rightarrow \P^1$
such that $f$ is $g^s$-equivariant.
\item[(5)]
One has $X_m = \F_e$ the Hirzebruch surface of degree $e \ge 0$
such that a/the ruling $\F_e \rightarrow \P^1$ is $g^{s}$-equivariant.
\item[(6)] One has $X = X_m = \P^2$.
So there are a $g$-equivariant blowup $\F_1 \rightarrow \P^2$ of a
$g | \P^2$-fixed point and the $g$-equivariant ruling $\F_1
\rightarrow \P^1$.
\end{itemize}
\end{lemma}

We now prove Lemma \ref{ps}. By Lemma \ref{deligne}, the Kodaira
dimension $\kappa(X) \le 0$. Consider first the case $\kappa(X) =
0$. Then $X$ contains only finitely many $(-1)$-curves and has a
unique smooth minimal model $X_m$. So $g | X$ descends to a
biregular action $g | X_m$. Rewrite $X = X_m$. Then $X$ is Abelian,
Hyperelliptic, $K3$ or Enriques. By Lemma \ref{k<1}, $X$ is an
abelian surface.

\par
Consider next the case where $X$ is an irrational ruled surface.
Then there is a $\P^1$-fibration $f : X \rightarrow E = \alb_X(X)$
with genral fibre $X_e \cong \P^1$, so that $p_a(E) = q(X) \ge 1$.
All rational curves (especially $(-1)$-curves) are contained in
fibres. $g |X$ permutes finitely many such $(-1)$-curves.
So we may assume that a positive power $g^s | X$ stabilizes every $(-1)$-curve
and let $X \rightarrow X_m$ be the $g^s$-equivariant blowdown
to a relatively minimal $\P^1$-fibration $f: X_m \rightarrow E$
where all fibres are $\P^1$. By the proof of Lemma \ref{n=1} and replacing
$s$, we may assume that $E$ is an elliptic curve and
$g^s | E$ is a translation of infinite order.
So Case(2) occurs.

\par
Consider the case where $X$ is a rational surface. So $\Pic(X) =
\NS(X)$. Assume that $g^* | \Pic(X)$ is finite, then $\Ker(\Aut(X)
\rightarrow \Aut(\Pic(X)))$ is infinite. Hence $X$ has only finitely many
$(-1)$-curves by Harbourne \cite{Hb} Proposition 1.3. As in the case
above, let $X \rightarrow X_m$ be a $g^s$-equivariant smooth
blowdown to a relatively minimal rational surface so that $X_m =
\P^2$, or $\F_e$ the Hirzebruch surface of degree $e \ge 0$. Note
that a/the ruling $\F_e \rightarrow \P^1$ is $g^{2s}$-equivariant
(the $''2''$ is to take care of the case $e = 0$ where there are two
rulings on $\F_e$). If $X_m = \P^2$ but $X \ne \P^2$, then we are
reduced to the case $\F_1$. If $X = X_m = \P^2$, the last case in
the lemma occurs (one trianglizes to see the fixed point).

\par
Assume that $X$ is rational and $g^* | \Pic(X)$ is infinite. By
\cite{Zh} Theorem 4.1 (or by Oguiso \cite{Og04} Lemma 2.8 and the
Riemann-Roch theorem applied to the $v$ and the adjoint divisor $K_X
+ v$ there as well as Fujita's uniqueness of the
Zariski-decomposition for pseudo-effective divisors like $v$ and
$K_X + v$ as formulated in Kawamata \cite{KMM} Theorem 7-3-1), there
is a $g$-equivariant smooth blowdown $X \rightarrow X_m$ such that
$K_{X_m}^2 = 0$, $-K_{X_m}$ is nef and $\kappa(X_m, -K_{X_m}) \ge 0$
(by the Riemann-Roch theorem). For simplicity, rewrite $X = X_m$.
\par
If $\kappa(X, -K_X) \ge 1$, then Case(4) occurs by the claim (and
the uniqueness of $f$) below.

\begin{claim}
Let $X$ be a smooth projective rational surface such that $-K_X$ is
nef, $K_X^2 = 0$ and $\kappa(X, -K_X) \ge 1$. Then $X$ is equipped
with a unique relatively minimal elliptic fibration $f : X
\rightarrow \P^1$ such that $-K_X$ is a positive multiple of a
fibre.
\end{claim}

We now prove the claim. Write $|-tK_X| = |M| + Fix$ for some $t
>>0$, so that $Fix$ is the fixed part. Note that $0 \le M^2 + M . Fix \le
(-tK_X)^2 = 0$. Thus $M \sim rF$ (linearly equivalent) with $|F|$ a
rational free pencil, noting that $q(X) = 0$. 
Also $M . Fix = 0$ and
$0 = (-tK_X)^2 = (Fix)^2$.
Hence $Fix$ is a rational multiple of $F$; see
Reid \cite{Re} page 36. 
Thus $-K_X$ is $\Q$-linearly equivalent to a positive multiple of $F$,
and $-K_X . F = 0 = F^2$. So $F$ is elliptic.
Since $K_X^2 = 0$ and by going to a relative
minimal model of the elliptic fibration and applying Kodaira's
canonical divisor formula there, we see that $f:= \Phi_{|F|} : X
\rightarrow C$ ($\cong \P^1$) is already relatively minimal. The
uniqueness of such $f$ again follows from Kodaira's this formula.
This proves the claim.

\par \vskip 1pc
We return to the proof of Lemma \ref{ps}. We still have to consider
the case where $X$ is rational, $g^* | \Pic(X)$ is infinite, $K_X^2 =
0$,  $-K_X$ is nef and $\kappa(X, -K_X) = 0$. We shall show that
Case(3) occurs. Take $x_0 \in X$ which does not lie on any negative
curve or {\it the} anit-pluricanonical curve in some $|-tK_X|$ or
the set $\cup_{r > 0} X^{g^r}$ of $g$-periodic points. Suppose the
contrary that the Zariski-closure $D(x_0)$ in Case(3) is not the
entire $X$. Then $D(x_0)$ is $1$-dimensional and we may assume that
a positive power $g^s$ stabilizes a curve $D_1 \ni x_0$ in $D(x_0)$.
By the choice of $x_0$, our $D_1^2 \ge 0$. If $-K_X$ is a rational
multiple of $D_1$, then we have $\kappa(X, -K_X) \ge 1$, a
contradiction. Otherwise, the class of $H : = D_1 - K_X$ is $(g^s)^*$-stable and $H^2
> 0$ by the Cauchy-Schwarz inequality (see the proof of \cite{BHPV}
IV (7.2)) or the Hodge index theorem, whence $g^s$ acts on
$H^{\perp} := \{L \in \Pic(X) \, | \, L . H = 0\}$ which is a
lattice with negative definite intersection form, so $(g^s)^* |
H^{\perp}$ and hence $g^* | \Pic(X)$ are periodic, a contradiction.
This proves Lemma \ref{ps}.

\par \vskip 1pc
The key for the 'splitting' of action below is from Lieberman \cite{Li}.

\begin{lemma}\label{split}
Let $X$ and $Y$ be projective manifolds. Suppose that the
second projection $f_Y: V = X \times Y \rightarrow Y$ is
$g$-equivariant. Then there is an action $g | X$ such that we can
write $g(x, y) = (g . x, g . y)$ for all $x \in X$, $y \in Y$, if
either Case $(1)$ or $(2)$ below occurs.
\begin{itemize}
\item[(1)]
The irregularity $q(X) = 0$, and $X$ is non-uniruled (or non-ruled).
\item[(2)]
$\dim X = \dim Y = 1$ and $\rank \, \NS_{\Q}(V) = 2$.
(These hold when one of $X, Y$ is $\P^1$,
or when $X, Y$  are non-isogenius elliptic curves).
\end{itemize}
\end{lemma}

\begin{proof}
As in Hanamura \cite{Hm} the proof of Theorem 2.3 there, we express
$g(x, y) = (\rho_g(y) . x, \ g.y)$ where $\rho_g : Y \rightarrow
\Aut(X)$ is a morphism. We consider Case (1). By \cite{Li} Theorem
3.12 and the proof of Theorem 4.9 there, the identity connected
component $\Aut_0(X)$ of $\Aut(X)$ is trivial, so $\Aut(X)$ is
discrete. Thus $\Imm(\rho_g)$ is a single point, denoted as $g | X
\in \Aut(X)$. The lemma is proved in this case.
\par
For Case(2), let $F$ be a fibre of $f_Y$. Then $g^*F = F'$ (another
fibre). Let $L$ be a fibre of the projection $f_X : V \rightarrow
X$. Since $\rank \, \NS_{\Q}(V) = 2$, we have $\NS_{\Q}(V) = \Q[F] +
\Q[L]$. Write $g^*L = aL + bF$. Then $1 = F . L = g^*F . g^*L = F .
g^*L = a$ and $0 = (g^*L)^2 = 2ab$ implies that $g^*L = L$ in
$\NS_{\Q}(V)$. Thus $g(L)$ is a curve with $g(L) . L = L^2 = 0$,
whence $g(L)$ is another fibre of $f_X$. The result follows. This
proves the lemma.
\end{proof}

We use Lieberman \cite{Li} Proposition 2.2 and Kodaira's lemma to
deduce the result below, though it is not needed in this paper (see
also Dinh-Sibony \cite{DS04} the proof of Theorem 4.6 there for a
certain case).

\begin{lemma}\label{nefbig}
Let $X$ be a projective manifold of dimension $n$, and $H \in
\Nef(X)$ a nef and big $\R$-Cartier divisor $($i.e. $H$ is nef and
$H^n > 0)$. Then $\Aut_H(X)/\Aut_0(X)$ is a finite group. Here
$\Aut_0(X)$ is the identity component of $\Aut(X)$, $\Aut_H(X) :=
\{\sigma \in \Aut(X) \ | \ \sigma^*H = H$ in $\NS_{\R}(X)\}$.
\end{lemma}

\begin{proof}
By Nakayama \cite{Na} II (3.16) and V (2.1), one may write $H = A +
D$ in $\NS_{\R}(X)$ with $A$ a $\Q$-ample divisor and $D$ an
effective $\R$-divisor. We follow the proof of Lieberman \cite{Li}
Proposition 2.2. For $\sigma \in \Aut_H(X)$, the volume of the graph
$\Gamma_{\sigma}$ is given by:
$$\vol(\Gamma_{\sigma}) = (A + \sigma^*A)^n \le (A + \sigma^*A)^{n-1}(H + \sigma^*H)
\le \dots \le (H + \sigma^*H)^n = 2^n H^n.$$
The rest of the proof is the same as \cite{Li}.
This proves the lemma.
\end{proof}

\par
The two results below will be used in the proofs in the next section.

\par
\begin{lemma}\label{Q-t}
The following are true.
\begin{itemize}
\item[(1)]
A $Q$-torus $Y$ does not contain any rational curve.
\item[(2)]
Let $f : X \cdots \to Y$ be a rational map from a normal projective variety $X$ with only
rational singularities (resp. log terminal singularities) to an abelian variety 
(resp. $Q$-torus) $Y$.
Then $f$ is a well-defined morphism.
\end{itemize}
\end{lemma}
\begin{proof}
(1) Let $T \to Y$ be a finite \'etale cover from a torus $T$. Suppose
the contrary that ${\bold P}^1 \to Y$ is a non-constant morphism.
Then $P := T \times_Y {\bold P}^1 \to {\bold P}^1$ is \'etale and
hence $P$ is a disjoint union of ${\bold P}^1$ by the simply
connectedness of ${\bold P}^1$. So the image in $T$ of $P$ is a
union of rational curves, contradicting the fact that a torus does
not contain any rational curve.

\par
(2) When $Y$ is an abelian variety, see \cite{Ka85} Lemma 8.1.

By Hironaka's resolution theorem, there is a birational proper morphism
$\sigma : Z \to X$ such that the composite $\tau = f \circ \sigma : Z \to X \cdots \to Y$ is a well defined morphism.
By Hacon-McKernan's solution to the Shokurov conjecture \cite{HM05a} Corollary 1.6,
every fibre of $\sigma$ is rationally chain connected
and is hence mapped to a point in $Y$, by (1). 
So for every ample divisor $H_Y \subset Y$, we have $\tau^*H_Y \sim_{\Q} \sigma^* L_X$
for a $\Q$-Cartier divisor $L_X \subset X$; see \cite{KM}, page 46, Remark (2).
Thus $\tau$ factors through $\sigma$, and
(2) follows.
\end{proof}

\begin{lemma}\label{ext}
Let $f : X \rightarrow Y$ be a $g$-equivariant surjective
morphism between projective manifolds and
with connected general fibre $F$. Assume the following conditions.
\begin{itemize}
\item[(1)]
All of $X, Y$ and $F$ are of positive dimension.
\item[(2)]
$Y$ is a $Q$-torus.
\item[(3)]
The Kodaira dimension $\kappa(X) = 0$
and $X$ has a good (terminal) minimal model $\bar{X}$, i.e.,
$\bar{X}$ has only terminal singularities and $sK_{\bar X} \sim 0$
for some $s > 0$.
\end{itemize}
Then there is a $g$-equivariant finite \'etale Galois extension $\tilde{Y} \to Y$
from a torus $\tilde{Y}$ such that the following are true.
\begin{itemize}
\item[(1)]
The composite $\bar{X} \cdots \to X \to Y$ extends to a morphism
with a general fibre $\bar{F}$. One has $s K_{\bar F} \sim 0$,
so $\bar{F}$ is a good terminal minimal model of $F$.
\item[(2)]
$X_1 := X \times_Y \tilde{Y}$ is birational to $\bar{X}_1 := \bar{F} \times \tilde{Y}$
over ${\tilde Y}$ with $sK_{{\bar X}_1} \sim 0$.
\item[(3)]
Denote by the same $G$ the group $\Gal(\tilde{Y}/Y)$
and the group $\id_X \times_Y \Gal(\tilde{Y}/Y) \le \Aut(X_1)$, and by the same
$g$ the automorphism $(g|X) \times_Y (g|\tilde{Y}) \in \Aut(X_1)$.
Then $g = g | X_1$ normalizes $G = G | X_1$.

\par \noindent
In the assertions $(4) - (7)$ below, suppose further that $q(F) = 0$.

\item[(4)]
$g|X_1 \in \Aut(X_1)$ induces a birational action $g$ on $\bar{X}_1$
with $g | \bar{X}_1 = (g | \bar{F}) \times (g | \tilde{Y})$,
where $g | {\bar F} \in \Bir(\bar{F})$ and $g | \tilde{Y} \in \Aut(\tilde{Y})$.
\par \noindent
In $(5) - (7)$ below, suppose in addition that $1 \le \dim F \le 2$.
\item[(5)]
Then $\dim F = 2$ and $\bar{F}$
is either a $K3$ or an Enriques. Further, the induced birational action
of $g = (g | \bar{F}) \times (g | \tilde{Y})$ on ${\bar{X}}_1 = \bar{F} \times \tilde{Y}$
is regular, i.e., $g | \bar{F} \in \Aut(\bar{F})$.
\item[(6)]
In $(5)$, $G = G | X_1 \le \Aut(X_1)$ induces a biregular action by $G
= G | {\bar{X}}_1 \le G_{\bar{F}} \times G_{\tilde Y}$ on ${\bar{X}}_1$
with $G_{\bar{F}} \le \Aut(\bar{F})$ and $G_{\tilde Y} = \Gal(\tilde{Y}/Y)$.
\item[(7)]
$g | X$ is neither rigidly parabolic nor of primitively positive entropy.
\end{itemize}
\end{lemma}

\begin{proof}
(1) follows from Lemma \ref{Q-t} and the fact that $K_{\bar F} = K_{\bar X} | {\bar F}$.
(2) is proved in Nakayama \cite{Na87} Theorem at page 427.
Indeed, for the $g$-equivalence of $\tilde{Y} \to Y$, by \cite{Na87},
(2) is true with an \'etale extension $Y' \to Y$. Let $T \to Y$ be an \'etale
cover of a torus $T$ of minimal degree. Then $g | Y$ lifts to $g | T$ as in
Beauville \cite{Be} \S 3. Now the projection $T':= T \times_Y Y' \to T$
is \'etale. So there is another \'etale cover $T'' \to T'$ such that
the composite $T'' \to T' \to T$ is just the multiplicative map
$m_{T''}$ for some $m > 0$. In particular, $T = m_{T''}(T'')$ is isomorphic to $T''$.
Clearly, the natural action $g | T''$ is compactible with the action $g | T$
via the map $m_{T''}$. Now the composition $\tilde{Y} := T'' \to Y$ is $g$-equivariant
and factors through $Y' \to Y$,
so that (2) is satisfied.
(3) is true because $g | X_1$ is the lifting of the action $g$ on $X = X_1/G$.
\par
We now assume $q(F) = 0$.
Assume that a group $\langle h \rangle$ acts on both $X_1$ and $\tilde{Y}$
compactibly with the cartesian projection $X_1 \to \tilde{Y}$.
For instance, we may take $\langle h \rangle$ to be a subgroup of
$G | X_1$ or $\langle g | X_1 \rangle$.
This $h$ acts birationally on $\bar{X}_1$.
To be precise, for $(x, y) \in \bar{X}_1$, we have $h . (x,
y) = (\rho_h(y) . x, \ h . y)$, where $\rho_h: \tilde{Y} \cdots \rightarrow
\Bir(\bar{F})$ is a rational map. By Hanamura \cite{Hm} (3.3), (3.10) and page 135,
$\Bir(\bar{F})$ is a disjoint union of abelian varieties of dimension
equal to $q(\bar{F}) = q(F) = 0$ (the first equality is true because
the singularities of $\bar{F}$ are terminal and hence rational).
Thus $\Imm (\rho_h)$ is a single element and denoted as $ h | \bar{F} \in \Bir(\bar{F})$.
So $h | \bar{X}_1 = (h | \bar{F}) \times (h | \tilde{Y})$.
\par
(4) follows by applying the arguments above to $h = g$.
For (5), suppose $\dim F = 1, 2$. Note that $\kappa(F) = 0 = q(F)$. So $F$
is birational to $\bar{F}$, a $K3$ or an Enriques, by the classification theory of surfaces.
Now (5) follows from the fact that $\Bir(S) = \Aut(S)$ for smooth minimal surface $S$,
by the uniqueness of surface minimal model.
The argument in the preceding paragraph also shows
$G = G | {\bar{X}}_1 \le G_{\bar{F}} \times G_{\tilde Y}$
with $G_{\bar{F}} \le \Bir(\bar{F}) = \Aut(\bar{F})$ and
$G_{\tilde Y} = \Gal(\tilde{Y}/Y) \, (= G) \le \Aut(\tilde{Y})$
(so that the two projections from $G|{\bar{X}}_1$ map onto $G_{\bar{F}}$ and $G_{\tilde Y}$,
respectively).
This proves (6).
\par
(7) Set $X' := {\bar{X}}_1/G$. Then $g$ acts on $X'$ biregularly
such that the pairs $(X', g)$ and $(X, g)$ are birationally equivalent,
$g | X_1$ being the lifting of $g | X$ and $X_1$ being birational to ${\bar{X}}_1$.
The projections $X' = {\bar{X}}_1/G \to \bar{F}/G_{\bar{F}}$ and $X' \to \tilde{Y}/G = Y$
are $g$-equivariant, since $g$ normalizes $G$.
\par
Suppose the contrary that $g | X$ is either rigidly parabolic,
or of primitively positive entropy. Then both $g | (\bar{F}/G_{\bar{F}})$
and $g | Y$ are rigidly parabolic by Lemma \ref{fib}
(applied to $g$-equivariant resolutions of both the source and targets of the projections).
In particular, $g | \bar{F}$ is parabolic by Lemma \ref{quot}
(applied to $g$-equivariant resolutions of the source and target of $\bar{F} \to \bar{F}/G_{\bar{F}}$).
\par
By Lemma \ref{k<1}, there is a unique
$g$-equivariant surjective morphism $\tau: \bar{F} \rightarrow \P^1$ (with fibre
$\bar{F}_p$) such that a positive power $g^s | \P^1 = \id$.
By (3), $g|\bar{F}$ normalizes $G_{\bar{F}}$.
So $(g|\bar{F})^*$ stabilizes the class of the nef divisor
$L := \sum_{h \in G_{\bar{F}}} h^*\bar{F}_p$. If $L$ is nef and big,
then, by Lemma \ref{nefbig}, a positive power of $g|\bar{F}$ is in $\Aut_0(\bar{F}) = (1)$, absurd.
Thus $L^2 = 0$, so $\bar{F}_p . h^*\bar{F}_p = 0$ and $G_{\bar{F}}$ permutes fibres of $\tau$.
Therefore, $\tau$ descends to a $g$-equivariant fibration
$\bar{F}/G_{\bar{F}} \to B \cong \P^1$ with $g^s|B =\id$, whence $g|(\bar{F}/G)$ is not
rigidly parabolic, absurd.
This proves (7) and also the lemma.
\end{proof}

\section{Results in arbitrary dimension; the proofs}

The results in Introduction follow from Theorem \ref{Thk} and three general results
below in dimension $\ge 3$.
\par
In the case of dimension $\le 3$, the good (terminal) minimal model program (as in
Kawamata \cite{Ka85}, or Mori \cite{Mo} \S 7) has been completed. So
in view of Theorems \ref{Thk}, \ref{Thk=0} and \ref{Thk=-1}, we are
able to describe the dynamics of $(X, g)$ in $(\ref{3-4k=0}) \sim
(\ref{3k=-1})$. See also Remark \ref{remark}.

\par
The result below is parallel to the conjecture (resp. theorem) of Demailly - Peternell - Schneider
(resp. Qi Zhang) to the effect that the Albanese map $\alb_X : X \to \Alb(X)$
is surjective whenever $X$ is a compact K\"ahler (resp. projective) manifold with $-K_X$ nef
(and hence $\kappa(X) = -\infty$).

\begin{theorem}\label{Thq}
Let $X$ be a projective complex manifold of $\dim X \ge 1$,
and $g \in \Aut(X)$. Suppose that the pair $(X, g)$ is either
rigidly parabolic or of primitively positive entropy $($see
$(\ref{entropy})$$)$. Then we have:
\begin{itemize}
\item[(1)]
The albanese map $\alb_X : X \rightarrow \Alb(X)$ is a
$g$-equivariant surjective morphism with connected fibres.
\item[(2)]
The irregularity $q(X)$ satisfies
$q(X) \le \dim X$.
\item[(3)]
$q(X) = \dim X$ holds if and only if $X$ is $g$-equivariantly
birational to an abelian variety.
\item[(4)]
$\alb_X : X \rightarrow \Alb(X)$ is a smooth surjective morphism if $q(X) < \dim X$;
see also $(\ref{smooth})$, $(\ref{fib})$.
\end{itemize}
\end{theorem}

\begin{theorem}\label{Thk=0}
Let $X$ be a projective complex manifold of dimension $n \ge
3$, with $g \in \Aut(X)$. Assume the following conditions.
\begin{itemize}
\item[(1)]
The Kodaira dimension $\kappa(X) = 0$ and the irregularity $q(X)
> 0$.
\item[(2)]
The pair $(X, g)$ is either rigidly parabolic or of primitively
positive entropy $($see $(\ref{entropy})$$)$.
\item[(3)]
$X$ has a good terminal minimal model
$($so $(3)$ is automatic if $n \le 3)$;
see $(\ref{ext})$, $(\ref{remark})$, \cite{Ka85} page $4$, \cite{Mo} \S $7$.
\end{itemize}
Then Case $(1)$ or $(2)$ below occurs.
\begin{itemize}
\item[(1)]
There are a $g$-equivariantly birational morphism $X \to X'$,
a pair $(\tilde{X'}, g)$ of a torus $\tilde{X'}$ and $g \in
\Aut(\tilde{X'})$, and a $g$-equivariant \'etale Galois cover
$\tilde{X'} \to X'$. In particular, $X'$ is a $Q$-torus.
\item[(2)]
There are a $g$-equivariant \'etale Galois cover $\tilde{X} \to X$,
a Calabi-Yau variety $F$ with $\dim F \ge 3$ $($see $(\ref{entropy}))$
and a birational map $\tilde{X} \cdots \to F \times \tilde{A}$
over $\tilde{A} := \Alb(\tilde{X})$.
Further, the biregular action $g | \tilde{X}$ is conjugate to
a birational action $(g | F) \times
(g | \tilde{A})$ on $F \times \tilde{A}$, where
$g | F \in \Bir(F)$ with the first dynamical degree
$d_1(g|F) = d_1(g|X)$, where $g | \tilde{A} \in \Aut(\tilde{A})$ is parabolic.
In particular, $\dim X \ge \dim F + q(X) \ge 4$.
\end{itemize}
\end{theorem}

\begin{theorem}\label{Thk=-1}
Let $X'$ be a projective complex manifold of dimension $n \ge
3$, with $g \in \Aut(X')$. Assume the following conditions $($see
$(\ref{remark}))$.
\begin{itemize}
\item[(1)]
The Kodaira dimension $\kappa(X') = -\infty$.
\item[(2)]
The pair $(X', g)$ is either rigidly parabolic or of primitively
positive entropy $($see $(\ref{entropy})$$)$.
\item[(3)]
The good terminal minimal model program is completed for varieties of
dimension $\le n$ $($so $(3)$ is automatic if $n \le 3)$;
see \cite{Ka85} p.$4$, \cite{Mo} \S $7$.
\end{itemize}
\par
Then there is a $g$-equivariant birational morphism $X \to X'$
from a projective manifold $X$ such that one of the cases below occurs.
\begin{itemize}
\item[(1)]
$X$ is a rationally connected manifold in the sense of \cite{Cp} and
\cite{KoMM}.
\item[(2)]
$q(X) = 0$.
The maximal rational connected fibration $\MRC_{X} : X
\rightarrow Z$ in the sense of $[$ibid$]$ is a well defined
$g$-equivariant surjective morphism. $Z$ is a weak Calabi-Yau
manifold with $\dim X > \dim Z \ge 3$.
\item[(3)]
$q(X) > 0$. There is a $g$-equivariant \'etale cover $\tilde{X}
\rightarrow X$ such that the surjective $g$-equivariant albanese map
$\alb_{\tilde X} : \tilde{X} \rightarrow \Alb(\tilde{X})$ coincides
with the maximal rationally connected fibration $\MRC_{\tilde X}$.
\item[(4)]
$q(X) > 0$.  There is a $g$-equivariant \'etale Galois cover $\tilde{X}
\rightarrow X$ such that the surjective albanese map $\alb_{\tilde
X} : \tilde{X} \rightarrow \tilde{A} := \Alb(\tilde{X})$ factors as
the $g$-equivariant $\MRC_{\tilde X} : \tilde{X} \rightarrow \tilde{Z}$
and $\alb_{\tilde Z} : \tilde{Z} \to \Alb(\tilde{Z}) = \tilde{A}$.
Further, there are a Calabi-Yau variety $F$ with $\dim F \ge 3$
$($see $(\ref{entropy}))$,
and a birational morphism $\tilde{Z} \to F \times \tilde{A}$
over $\tilde{A}$,
such that the biregular action $g | \tilde{Z}$ is conjugate to
a birational action $(g | F) \times (g | \tilde{A})$ on $F \times \tilde{A}$, where
$g | F \in \Bir(F)$, where $g | \tilde{A} \in \Aut(\tilde{A})$ is parabolic.
Also $\dim X > \dim F + q(X) \ge 4$.
\end{itemize}
\end{theorem}

\begin{remark}\label{remark}
\end{remark}

\par \noindent
(a) By the proof, the condition (3) in
Theorem \ref{Thk=-1} can be weakened to:
\begin{itemize}
\item[(3)']
$X'$ is uniruled. For every projective variety $Z$ dominated by a proper subvariety
($\ne X'$) of $X'$, if the Kodaira dimension $\kappa(Z) = -\infty$
then $Z$ is uniruled, and if $\kappa(Z) = 0$ then $Z$
has a good terminal minimal model $Z_m$
(i.e., $Z_m$ has only terminal singularities and $sK_{Z_m} \sim 0$
for some $s > 0$); see \cite{Ka85} page 4, and \cite{Mo} \S 7.
\end{itemize}

\par \noindent
In dimension three, the good terminal minimal model program has been
completed; see \cite{KMM} and \cite{KM}. For the recent break
through on the minimal model program in arbitrary dimension, we
refer to Birkar - Cascini - Hacon - McKernan \cite{BCHM}.
\par \noindent
(b) The good minimal model program also implies the equivalence of
the Kodaira dimension $\kappa(X) = -\infty$ and the uniruledness of
$X$. It is known that the uniruledness of $X$ always implies
$\kappa(X) = -\infty$ in any dimension.
\par \noindent
(c) The birational automorphisms $g | F$ in Theorems \ref{Thk=0} and \ref{Thk=-1}
are indeed isomorphisms in codimenion $1$; see \cite{Hm} (3.4).
\par \noindent
(d) See \cite{NZ} Theorem B for a stonger result for the case of $K_X \equiv 0$.

\par \vskip 0.5pc
As consequences of Theorems \ref{Thk=0} and
\ref{Thk=-1} for all dimension $\ge 3$ and as illustrations, we have
the simple 3-dimensional formulations of them as in
$(\ref{3-4k=0}) \sim (\ref{3k=-1})$ below.

\par \vskip 1pc
The result below says that the dynamics on an irregular threefold
of Kodaira dimension $0$, are essentially the dynamics of a
torus.

\begin{corollary}\label{3-4k=0}
Let $X'$ be a smooth projective complex threefold,
with $g \in \Aut(X')$. Assume that the Kodaira dimension $\kappa(X') = 0$,
irregularity $q(X')$ $> 0$, and the pair $(X', g)$ is either rigidly parabolic or
of primitively positive entropy; see
$(\ref{entropy})$.
\par
Then there are a $g$-equivariant birational morphism $X' \to X$,
a pair $(\tilde{X}, g)$ of a torus $\tilde{X}$ and $g \in \Aut(\tilde{X})$,
and a $g$-equivariant \'etale Galois cover $\tilde{X} \to X$.
In particular, $X$ is a $Q$-torus.
\end{corollary}

The result below shows that the dynamics on a threefold of Kodaira
dimension $-\infty$ are (or are built up from) the dynamics on a
rationally connected threefold (or on a rational surface and that on
a $1$-torus).

\begin{theorem}\label{3k=-1}
Let $X$ be a smooth projective complex threefold, with $g \in
\Aut(X)$. Assume that $\kappa(X) = -\infty$, and the pair $(X,
g)$ is either rigidly parabolic or of primitively positive entropy
$($see $(\ref{entropy})$$)$.
Then we have:
\begin{itemize}
\item[(1)]
If $q(X) = 0$ then $X$ is rationally connected in the sense of
\cite{Cp} or \cite{KoMM}.
\item[(2)]
Suppose that $q(X) \ge 1$ and the pair $(X, g)$ is of primitively
positive entropy. Then $q(X) = 1$ and the albanese map $\alb_X: X
\rightarrow \Alb(X)$ is a smooth surjective morphism with every fibre $F$ a
smooth projective rational surface of Picard number $\rank \,
\Pic(F) \ge 11$.
\end{itemize}
\end{theorem}

\begin{setup}
{\bf Proof of Theorem \ref{Thk}}.
\end{setup}

The assertion (1) follows from Lemma \ref{deligne}. For (2), in view
of (1), we may assume that $(X, g)$ is of imprimitively positive
entropy. Then the assertion (2) follows from Lemma \ref{3nef}. This
proves Theorem \ref{Thk}.

\begin{setup} {\bf Proof of Theorem \ref{Thq}}.
\end{setup}

We may assume that $q(X) > 0$. By the universal property of $A:=
\Alb(X)$, every $h \in \Aut(X)$ descends, via
the albanese map $\alb_X : X \rightarrow A$, to some $h|A \in \Aut_{\variety}(A)$.
By Lemma \ref{deligne}, $\kappa(\alb_X(X)) \le 0$. Thus,
by Ueno \cite{Ue} Lemma 10.1, $\kappa(\alb_X(X)) = 0$ and $\alb_X(X)
= A = \Alb(X)$, i.e., $\alb_X$ is surjective. Let $X \rightarrow X_0
\rightarrow A$ be the Stein factorization with $X_0 \rightarrow
A$ a finite surjective morphism from a normal variety $X_0$, and $X
\rightarrow X_0$ having connected fibres. Note that $\kappa(X_0) \ge
0$ by the ramification divisor formula for (the resolution of the
domain of) $X_0 \rightarrow A$ as in Iitaka \cite{Ii82} Theorem 5.5.
So by Lemma \ref{deligne}, $\kappa(X_0) = 0$. By the result of
Kawamata-Viehweg as in Kawamata \cite{Ka81} Theorem 4, $X_0
\rightarrow A$ is \'etale, so $X_0$ is an abelian variety too. By the
universal property of $A = \Alb(X)$, we have $X_0 = A$. Thus, $X
\rightarrow A = X_0$ has connected fibres. Theorem \ref{Thq} (1) is
proved. Now Theorem \ref{Thq} (2) and (3) follow from (1). If $q(X)
< \dim X$ then $g | A$ is rigidly parabolic by Lemma \ref{fib}; so
Theorem \ref{Thq} (4) follows from Lemma \ref{smooth}. This proves
Theorem \ref{Thq}.

\begin{setup} {\bf Albanese variety.}
\end{setup}
For a projective variety $Z$, we denote by $A(Z)$ or $\Alb(Z)$
the albanese variety $\Alb(Z')$ with $Z' \to Z$ a proper resolution.
This definition is independent of the choice of $Z'$,
and $A(Z)$ depends only on the birational equivalence class of $Z$.
If $Z$ is log terminal, then the composition $Z \cdots \to Z' \to A(Z)$
is a well defined morphism; see Lemma \ref{Q-t}. 

\par
\begin{setup} {\bf Proof of Theorem \ref{Thk=0}.}
\end{setup}
By Theorem \ref{Thq}, we
may assume that $q(X) < \dim X$, so $g | A(X)$ is rigidly parabolic
by Lemma \ref{fib}. The albanese map $\alb_{X}: X \rightarrow A(X)$
has connected fibre $F_1$ and is smooth and surjective; see Theorem
\ref{Thq}. 

\par
By the assumption, $X$ has a good terminal minimal model $\bar{X}$
with $s K_{\bar X} \sim 0$ for some $s > 0$.
We apply Lemma \ref{ext} to $\alb_X : X \to Y_1 := A(X)$.
Then there is a $g$-equivariant \'etale Galois extension $\tilde{Y}_1 \to Y_1$
from a torus $\tilde{Y}_1$,
such that $X_1 := X \times_{Y_1} \tilde{Y}_1$ is birational to
$\bar{X}_1 := \bar{F}_1 \times \tilde{Y}_1$ over $\tilde{Y}_1$, with $\bar{F}_1$ a good terminal
minimal model of $F_1$ and $s K_{{\bar F}_1} \sim 0$.
Also $g | X_1$ normalizes $G_1 | X_1$ ($\cong \Gal(\tilde{Y}_1/Y_1)$).

\par
Assume that $0 < q(F_1) < \dim F_1$.
By \cite{Ka81} Theorem 1, $\alb_{X_1} : X_1 \to A(X_1)$ is a surjective morphism
with connected smooth general fibre $F_2$.
By the universal property
of the albanese map, $\alb_{X_1} : X_1 \rightarrow A(X_1)$
is $\langle g, G_1 \rangle$-equivariant.
Both of the natural morphisms $X = X_1/G_1 \rightarrow Y_2:= A(X_1)/G_1$
and $A(X_1) \to Y_2$ are $g$-equivariant and surjective. Since $G_1$
acts freely on $\tilde{Y}_1$ and $A(X_1) = A(\bar{X}_1) = A(F_1) \times \tilde{Y}_1$,
the latter map is \'etale. By the same
reason, every general fibre of $X \rightarrow Y_2$ can be identified with a
fibre $F_2$, so it is connected.

\par
We apply Lemma \ref{ext} to $X \to Y_2$.
Then there is a $g$-equivariant \'etale Galois extension $\tilde{Y}_2 \to Y_2$
from a torus $\tilde{Y}_2$,
such that $X_2 := X \times_{Y_2} \tilde{Y}_2$ is birational to
$\bar{X}_2 := \bar{F}_2 \times \tilde{Y}_2$ with $\bar{F}_2$ a good terminal
minimal model of $F_2$ and $s K_{{\bar F}_2} \sim 0$.
Also $g | X_2$ normalizes $G_2 | X_2$ ($\cong \Gal(\tilde{Y}_2/Y_2)$).

\par
If $0 < q(F_2) < \dim F_2$, we can consider
$X \to Y_3 := A(X_2)/G_2$.
Continue this process, we can define $X \to Y_{i+1} := A(X_i)/G_i$
with $G_i \cong \Gal(\tilde{Y}_i/Y_i)$ the Galois group
of the \'etale Galois extension $\tilde{Y}_i \to Y_i$ from a torus $\tilde{Y}_i$,
such that $X_i := X \times_{Y_i} \tilde{Y}_i$ is birational to
$\bar{X}_i := \bar{F}_i \times \tilde{Y}_i$ with $s K_{{\bar X}_i} \sim 0$,
where $\bar{F}_i$ a good terminal model of a general fibre $F_i$
of $X \to Y_i$ (and also of $X_i \to \tilde{Y}_i$
and $X_{i-1} \to A(X_{i-1})$).
\par
Note that $q(F_i) \le \dim F_i$ because $\kappa(F_i)
= \kappa(\bar{F}_i) = 0$ (see \cite{Ka81} Theorem 1).
Also $\dim X \ge \dim Y_{i+1} = \dim Y_i + q(F_i)$.
So there is an $m \ge 1$ such that $q(F_m)$ equals either $0$ or $\dim
F_m$.
\par
Consider the case where $q(F_m) = 0$ and $\dim F_m > 0$.
Then $\Alb(X_m) = \Alb(\bar{X}_m) = \tilde{Y}_m$, and
by Lemma \ref{ext},
Theorem \ref{Thk=0} Case(2) occurs with $\tilde{X} = X_m$,
$F = \bar{F}_m$ and $T = \tilde{Y}_m$. Indeed, for the second part of
Theorem \ref{Thk=0} (2), since $X \to Y_m$ is $g$-equivariant,
$g | Y_m$ is rigidly parabolic by Lemma \ref{fib} and hence
$g | T$ is parabolic by Lemma \ref{quot} (with $d_1(g|T) = 1$).
The first dynamical degrees satisfy $d_1(g|X)
= d_1(g | \tilde{X}) = d_1((g|F) \times (g|T)) = d_1(g | F)$
by Lemma \ref{quot}, Guedj \cite{Gu03} Proposition 1.2 and the K\"unneth formula
for $H^2$ as in Griffiths-Harris \cite{GH} page 58;
see also \cite{Di} Proposition 5.7.
Also $\dim X = \dim F + \dim \tilde{Y}_m \ge \dim F + \dim Y_1 = \dim F + q(X)$.
\par
Consider the case ($q(\bar{F}_m) =$) $q(F_m) = \dim F_m$. Then $q(X_m) = \dim X_m$. By
Kawamata \cite{Ka81} Theorem 1, the albanese map $X_m \rightarrow
\tilde{X}' := A(X_m) = A(F_m) \times \tilde{Y}_m$ is a $\langle g,
G_m\rangle$-equivariant birational surjective  morphism. It induces
a $g$-equivariant birational morphism $X = X_m/G_m \rightarrow X'
:= \tilde{X}'/G_m$, since $g$ normalizes $G_m$ as in Lemma \ref{ext}.
Also $G_m$ acts freely on $\tilde{Y}_m$, and hence the quotient map
$\tilde{X}' \rightarrow X'$ is \'etale. Note that $X' \times_{Y_m}
\tilde{Y}_m \cong \tilde{X}'$ over $\tilde{Y}_m$, since both sides
are finite (\'etale) over $X'$ and birational to each other,
by the construction of $X_m$.
Thus
Case(1) of Theorem \ref{Thk=0} occurs with the \'etale Galois cover
$\tilde{X}' \rightarrow X'$. This proves Theorem \ref{Thk=0}.

\begin{setup} {\bf Proof of Theorem \ref{Thk=-1}.}
\end{setup}
Let $\MRC_{X'} : X' \cdots \rightarrow Z$ be a maximal rationally
connected fibration; see \cite{Cp}, or \cite{Ko} IV Theorem 5.2. The
construction there, is in terms of an equivalence relation, which is
preserved by $g | X'$. So we can replace $(X', g)$ by a
$g$-equivariant blowup $(X, g)$ such that $\MRC_{X} : X
\rightarrow Z$ is a well defined $g$-equivariant surjective morphism with
general fibre rationally connected,
$g | X \in \Aut(X)$, and $X$, $Z$ projective manifolds; see
Hironaka \cite{Hi}. Further, $Z$ is non-uniruled by
Graber-Harris-Starr \cite{GHS} (1.4). The natural
homomorphism $\pi_1(X) \rightarrow \pi_1(Z)$ is an isomorphism; see
Campana \cite{Cp} or Koll\'ar \cite{Ko}. So $q(X) = q(Z)$. If $\dim Z
= 0$, then Case(1) of the theorem occurs.
\par
Consider the case $\dim Z > 0$.
Since $X'$ is uniruled by the assumptions of the theorem
(see Remark \ref{remark}), $\dim Z < \dim X'$.
Since $Z$ is non-uniruled, we have $\kappa(Z) \ge 0$
by the assumption. So $\kappa(Z) = 0$ by Lemma \ref{deligne}.
\par
Now $g | Z$ is rigidly parabolic by Lemma \ref{fib}. If $q(Z) = 0$
then $\dim Z \ge 3$ because $\kappa(Z) = 0$ and by Lemma \ref{ps}.
So Case(2) of the theorem occurs.
\par
Suppose $q(Z) > 0$. Since an abelian variety contains no rational
curves, $\alb_X :$ $X \rightarrow A := \Alb(X)$ factors as $\MRC_X:
X \rightarrow Z$ and $\alb_Z : Z \rightarrow \Alb(Z)$ $= A$;
see Lemma \ref{Q-t} and \cite{Ka81} Lemma 14.
By Lemma \ref{fib}, $g | A$ is rigidly parabolic. Also $\alb_X$ and $\alb_Z$
are smooth and surjective with connected fibres by Theorem \ref{Thq}.

\par
We apply Theorem \ref{Thk=0} to $(Z, g)$, so two cases there occur;
in the first case there, we may assume that $K_Z$ is torsion after replacing
$Z$ by its $g$-equivariant blowdown.
Let $\tilde{Z} \rightarrow Z$ be the $g$-equivariant \'etale Galois extension
as there. So either $\tilde{Z}$ is a torus, or
$\tilde{Z} \to F \times \tilde{A}$ is a well defined birational morphism
over $\tilde{A} := \Alb(\tilde{Z})$
(after replacing $Z$ and $X$ by their $g$-equivariant blowups),
with $q(F) = 0$ etc as described there.
Set $\tilde{X} := X
\times_Z \tilde{Z}$. Then the projection $\tilde{X} \rightarrow
\tilde{Z}$ coincides with $\MRC_{\tilde X}$. So $\alb_{\tilde X} :
\tilde{X} \rightarrow \Alb(\tilde{X}) = \tilde{A}$ factors as
$\tilde{X} \rightarrow \tilde{Z}$ and $\alb_{\tilde Z} : \tilde{Z}
\rightarrow \tilde{A}$. If $\tilde{Z}$ is a torus,
then Case(3) of Theorem \ref{Thk=-1} occurs. In the situation $\tilde{Z} \to F
\times \tilde{A}$, Case(4) of
Theorem \ref{Thk=-1} occurs in view of Theorem \ref{Thk=0}. This
proves Theorem \ref{Thk=-1}.

\begin{setup} {\bf Proofs of Theorems \ref{3k} $\sim$ \ref{Thsalem} and \ref{3k=-1}, and
Corollary \ref{3-4k=0}.}
\end{setup}
Corollary \ref{3-4k=0} and Theorem \ref{3k=-1} (1) follow respectively from
Theorems \ref{Thk=0} and \ref{Thk=-1}, while Theorem \ref{3k}
follows from Lemma \ref{3nef}, Corollary \ref{3-4k=0}, Theorem
\ref{3k=-1}, and Lemma \ref{salem} applied to $\alb_X$.
Theorem \ref{Thg} follows from Theorem \ref{3k} (and its proof).
Theorem \ref{Thsalem} follows from Lemma \ref{3nef}, Corollary \ref{3-4k=0},
Theorem \ref{3k=-1}, the proof of Lemma \ref{salem}, and Lemma
\ref{quot}. See Remark \ref{remark}.
\par
For Theorem \ref{Thp}, by Theorems \ref{Thk} and
\ref{3k=-1}, we have only to consider the case in
Theorem \ref{3k=-1} (2). But then $\pi_1(X) = \pi_1(\Alb(X)) = \Z^{\oplus 2}$
since general (indeed all) fibres of $\alb_X : X \to \Alb(X)$
are smooth projective rational surfaces (see \cite{Cp} or \cite{KoMM}).

\par
We now prove Theorem \ref{3k=-1} (2) directly. We follow the proof
of Theorem \ref{Thk=-1}. Let $\MRC_{X}: X \cdots \rightarrow Y$ be
a maximal rationally connected fibration, where $\kappa(Y) \ge 0$. Replacing $Y$ by
a $g$-equivariant modification, we may assume that $Y$ is smooth and minimal.
Since $\kappa(X) = -\infty$,
our $X$ is uniruled (see Remark \ref{remark}). So $\dim Y < \dim X$. Our
$\alb_{X} : X \rightarrow A: = \Alb(X)$ is smooth and surjective (with
connected fibre) and factors as $\MRC_X : X \cdots \rightarrow Y$ and
$\alb_Y : Y \rightarrow \Alb(Y) = A$; also $3 = \dim X > \dim Y \ge
q(Y) = q(X)
> 0$ (by the assumption of Theorem \ref{3k=-1} (2)), $\kappa(Y) = 0$
and $g | Y$ and $g | A$ are rigidly parabolic; see the proof of
Theorem \ref{Thk=-1}. Theorem \ref{3k=-1} (2) follows from the two
claims below.

\begin{claim}
In Theorem \ref{3k=-1} (2), $q(X) = 1$.
\end{claim}

\begin{proof}
Suppose the contrary that $q(X) \ge 2$. Then $\dim Y = q(Y) = q(X) =
2$, so $Y = A$ and $\alb_X = \MRC_X$, by Theorem \ref{Thq}
and the proof of Theorem \ref{Thk=-1}.
By Theorem \ref{Thq}, $\alb_X : X \rightarrow A$ is
surjective with every fibre $F$ a smooth projective curve. $F$ is a
rational curve by the definition of $\MRC_X : X \rightarrow Y = A$.
Take a nef $L = v_{(g | X)} \in \Nef(X)$ as in Lemma \ref{PF} such
that $g^* L = d_1 L$ with $d_1 = d_1(g | X) > 1$. By Lemma
\ref{Abel}, either $g^s | A$ is a translation and we let $C$ be a
very ample divisor on $A$, or $g^s(C) = C + t_0$ for an elliptic
curve $C$ and $t_0 \in A$. Rewrite $g^s$ as $g$ and we always have
$g^*C \equiv C$ in $N_1(A) = \NS_{\R}(A)$.
\par
Let $X_C \subset X$ be the inverse of $C$, a Hirzebruch surface.
Then the restriction
$\alb_{X} | X_C : X_C \rightarrow C$ is a ruling. Note that $g^*F =
F$ in $N_1(X)$. Also $g^* = \id$ on $H^6(X, \R) \cong \R$. So $L . F
= g^*L . g^*F = d_1 L . F$, whence $L . F = 0$. By the adjunction
formula, $K_{X_C} = (K_X + X_C) | X_C = K_X | X_C + e F$ with the
scalar $e = C^2 \ge 0$. Note that $\R \cong H^6(X, \R) \ni L . K_X .
X_C = g^*L . g^*K_X . g^*X_C = d_1 L . K_X . X_C$, whence $0 = L .
K_X . X_C = (L | X_C) . (K_X | X_C) = (L | X_C) . K_{X_C}$ since $L
. F = 0$. Now $(L|X_C) . F = L . F = 0$ and $(L | X_C) . K_{X_C} =
0$ imply that $L . X_C = L | X_C = 0$ in the lattice $\NS_{\R}(X_C)$
because the fibre $F$ and $K_{X_C}$ span this lattice. Hence $X_C$
and $L$ are proportional in $\NS_{\R}(X)$ by Lemma \ref{DS} (1),
noting that any curve like $C$ in the abelian surface $A$ is nef and
hence $X_C$ is nef. But $g^* X_C = X_C$ while $g^*L = d_1 L$ with
$d_1 > 1$, so $X_C$ and $L$ are not proportional in $\NS_{\R}(X)$.
Thus, the claim is true.
\end{proof}

\begin{claim}
In Theorem \ref{3k=-1} (2), every fibre $X_a$ of $\alb_X : X \rightarrow A =
\Alb(X)$ is a smooth projective rational surface
(so $\NS(X) = \Pic(X)$) with non-big $-K_{X_a}$.
Further, $\rank \, \Pic(X_a) \ge 11$.
\end{claim}

\begin{proof}
If $\dim Y > \dim A = q(X) = 1$, then $Y$ is a surface with
$\kappa(Y) = 0$, $q(Y) = 1$ and $g | Y$ rigidly parabolic, which
contradicts Lemma \ref{ps}. Thus $\dim Y = 1 = q(Y)$, so $Y =
\Alb(Y) = A$ and $\alb_X = \MRC_X$ by Theorem \ref{Thq}. 
Therefore, every fibre $F = X_a$ of $X \rightarrow A$
is a smooth projective rational surface;
see Koll\'ar \cite{Ko} IV Theorem 3.11.
\par
Assume the contrary that $-K_F$ is big or $\rank \, \Pic(F) \le 10$
(i.e., $K_F^2 \ge 0$)
and we shall derive a contradiction.
If $K_F^2 \ge 1$, then $-K_F$ is big by the Riemann-Roch theorem applied
to $-nK_F$. Thus we may assume that either $-K_F$ is big or $K_F^2 = 0$
and shall get a contradiction.
\par
As in the proof of the previous claim, for
$K_F = K_X | F$ and $L := v_{(g|X)}$, we have $0 = L . K_X . F =
(L | F) . K_F$, so $L | F \equiv c K_F = c K_X | F$ for some scalar $c$
by the Hodge index theory; see the proof of \cite{BHPV} IV (7.2).
If $c \ne 0$, applying $g^*$, we get $d_1(g) = 1$, absurd.
Hence $c = 0$ and $L . F = 0$. Then $L \equiv e F$ for some scalar $e > 0$
by Lemma \ref{DS}. Applying $g^*$, we get the same contradiction.
This proves the claim
and also Theorem \ref{3k=-1}.
\end{proof}

\end{large}

\end{document}